\documentclass[a4paper]{amsart}

\usepackage{amsmath}
\usepackage{amssymb}

\numberwithin{figure}{section}

\newtheorem{theorem}{Theorem}[section]
\newtheorem{proposition}[theorem]{Proposition}
\newtheorem{lemma}[theorem]{Lemma}
\newtheorem{corollary}[theorem]{Corollary}
\theoremstyle{definition}
\newtheorem{defprop}[theorem]{Definition-Proposition}
\newtheorem{definition}[theorem]{Definition}
\newtheorem{example}[theorem]{Example}
\newtheorem{conventions}[theorem]{Conventions}
\newtheorem{setting}[theorem]{Setting}
\theoremstyle{remark}
\newtheorem{remark}[theorem]{Remark}

\usepackage{url}
\usepackage{hyperref}
\hypersetup{%
  hidelinks,%
  hyperfigures,%
  pagebackref,%
  pdfinfo={%
    Title={Reduction of tau-tilting modules and torsion pairs},%
    Author={Gustavo Jasso},%
    Keywords={tau-tilting theory; torsion pairs; tilting theory;
      cluster-tilting; silting object; mutation}%
  }%
}

\usepackage{my-commands}

\begin{document}

\title{Reduction of $\tau$-tilting modules and torsion pairs}  

\author[G. Jasso]{Gustavo Jasso}

\address{Graduate School of Mathematics, Nagoya
  University. Nagoya, Japan}

\email{jasso.ahuja.gustavo@b.mbox.nagoya-u.ac.jp}



\thanks{The author wishes to thank his Ph.D. supervisor, Prof. O. Iyama, %
for suggesting these problems and to acknowledge the constant %
support and valuable comments received from him during the %
preparation of this article. %
This acknowledgement is extended to D. Yang for helpful %
discussions on the subject and to R. Kase for pointing out the %
connection with known results for hereditary algebras. %
%%% Local Variables: 
%%% mode: plain-tex
%%% TeX-master: "master"
%%% End: 
}

\maketitle

\begin{abstract}
The class of support $\tau$-tilting modules was introduced recently by Adachi, Iyama and Reiten.  These modules complete the class of tilting modules from the point of view of mutations. Given a finite dimensional algebra $A$, we study all basic support $\tau$-tilting $A$-modules which have given basic $\tau$-rigid $A$-module as a direct summand. We show that there exist an algebra $C$ such that there exists an order-preserving bijection between these modules and all basic support  $\tau$-tilting $C$-modules; we call this process $\tau$-tilting reduction. An important step in this process is the formation of $\tau$-perpendicular categories which are analogs of ordinary perpendicular categories. Finally, we show that $\tau$-tilting reduction is compatible with silting reduction and 2-Calabi-Yau reduction in appropiate triangulated categories.
\end{abstract}

%%% Local Variables: 
%%% mode: latex
%%% TeX-master: "master"
%%% End: 

\tableofcontents

\section{Introduction}

Let $A$ be a finite dimensional algebra over a field.
Recently, Adachi, Iyama and Reiten introduced in
\cite{adachi_tau-tilting_2012} a generalization of classical
tilting theory, which they called \emph{$\tau$-tilting theory}.
Motivation to study $\tau$-tilting theory comes from various
sources, the most important one is mutation of tilting modules. 
Mutation of tilting modules has its origin in
Bernstein-Gelfand-Ponomarev reflection functors
\cite{bernstein_coxeter_1973}, which were later generalized by
Auslander, Reiten and Platzeck with the introduction
of APR-tilting modules \cite{auslander_coxeter_1979}, which are
obtained by replacing a simple direct summand of the tilting
$A$-module $A$.
Mutation of tilting modules was introduced in full generality by
Riedtmann and Schofield in their combinatorial study of
tilting modules \cite{riedtmann_simplicial_1991}. 
Also, Happel and Unger showed in \cite{happel_partial_2005}
that tilting mutation is intimately related to the partial order
of tilting modules induced by the inclusion of the associated
torsion classes. 

We note that one limitation of mutation of tilting modules is that
it is not always possible. 
This is the motivation for the introduction of $\tau$-tilting
theory. 
Support $\tau$-tilting (resp. $\tau$-rigid) $A$-modules are a
generalization of tilting (resp. partial-tilting) $A$-modules
defined in terms of the Auslander-Reiten translation, see
Definition \ref{def:sttilt}.
Support $\tau$-tilting modules can be regarded as a ``completion''
of the class of tilting modules from the point of view of
mutation.
In fact, it is shown in \cite[Thm. 2.17]{adachi_tau-tilting_2012}
that a basic almost-complete support $\tau$-tilting $A$-module is
the direct summand of exactly two basic support $\tau$-tilting
$A$-modules.
This means that mutation of support $\tau$-tilting $A$-modules is 
always possible.

It is then natural to consider more generally all support
$\tau$-tilting $A$-modules which have a given $\tau$-rigid
$A$-module $U$ as a direct summand.
% It is worth noting that there is an analog of Bongartz completion
% for $\tau$-rigid modules
% \cite[Prop. 2.9]{adachi_tau-tilting_2012}.
% Namely, for every $\tau$-rigid $A$-module $U$ there exists a
% distinguished $\tau$-tilting $A$-module $T_U$ having $U$
% as a direct summand.
Our main result is the following bijection:

\begin{theorem}
  [see Theorem \ref{thm:tau-tilting-reduction} for details]
  \label{ithm:tau-tilting-reduction}
  Let $U$ be a basic $\tau$-rigid $A$-module.
  Then there exists a finite dimensional algebra $C$ such that
  there is an order-preserving bijection between the
  set of isomorphism classes of basic support $\tau$-tilting
  $A$-modules which have $U$ as a direct summand and the set of 
  isomorphism classes of all basic support $\tau$-tilting
  $C$-modules.
  We call this process \emph{$\tau$-tilting reduction}.
\end{theorem}

% We note that $\tau$-tilting reduction is compatible with mutations
% of support $\tau$-tilting modules, see Corollary
% \ref{cor:mutations}. 
As a special case of Theorem \ref{ithm:tau-tilting-reduction} we
obtain an independent proof of
\cite[Thm. 2.17]{adachi_tau-tilting_2012}.

\begin{corollary}[Corollary \ref{cor:2-complements-coro}]
  Every almost-complete support $\tau$-tilting $A$-module is the
  direct summand of exactly two support $\tau$-tilting
  $A$-modules. 
\end{corollary}

If we restrict ourselves to hereditary algebras, then Theorem
\ref{ithm:tau-tilting-reduction} gives the following
improvement of \cite[Thm. 3.4]{happel_links_2010}, where $U$ is
assumed to be faithful.

\begin{corollary}[Corollary \ref{cor:tilting-reduction}]
  \label{ithm:tilting-reduction}
  Let $A$ be a hereditary algebra and $U$ be a basic
  partial-tilting $A$-module.
  Then there exists a hereditary algebra $C$ such that there is an
  order-preserving bijection between the set of
  isomorphism classes of basic support tilting $A$-modules which
  have $U$ as a direct summand and the set of isomorphism classes
  of all basic support tilting $C$-modules. 
\end{corollary}

Now we explain a category equivalence which plays a fundamental
role in the proof of Theorem \ref{ithm:tau-tilting-reduction}, and
which is of independent interest.
Given a $\tau$-rigid module $U$, there are two torsion pairs in
$\mod A$ which are naturally associated to $U$. 
Namely, $(\Fac U,U^\perp)$ and $(\lperp{(\tau U)},\Sub(\tau U))$.
We have the following result about the category 
$\lperp{(\tau U)}\cap U^\perp$, which is an analog of the
perpendicular category associated with $U$ in the sense of
\cite{geigle_perpendicular_1991}, see Example
\ref{ex:tau-perpendicular}.

\begin{theorem}[Theorem \ref{thm:U-modC}]
  \label{ithm:U-modC}
  With the hypotheses of Theorem \ref{ithm:tau-tilting-reduction}, let
  $T_U$ be the Bongartz completion of $U$ in $\mod A$. Then,
  the functor $\Hom_A(T_U,-):\mod A\to \mod(\End_A(T_U))$ induces 
  an equivalence of exact categories
  \[ 
  F:\lperp(\tau U)\cap U^\perp\longrightarrow\mod C. 
  \]
\end{theorem}

It is shown in \cite[Thm. 2.2]{adachi_tau-tilting_2012} that
basic support $\tau$-tilting $A$-modules are precisely the
$\Ext$-progenerators of functorially finite torsion classes in 
$\mod A$. 
The proof of Theorem \ref{ithm:tau-tilting-reduction}
makes heavy use of the bijection between functorially finite
torsion classes in $\mod A$ and basic support $\tau$-tilting
$A$-modules. 
The following result extends the bijection given in Theorem
\ref{ithm:tau-tilting-reduction}, as the torsion classes involved do
not need to be functorially finite:

\begin{theorem}[Theorem \ref{thm:tp-reduction}]
  With the hypotheses of Theorem \ref{ithm:tau-tilting-reduction},
  the map
  \[
  \T\mapsto F(\T\cap U^\perp)
  \]
  induces a bijection between torsion classes $\T$ in $\mod A$
  such that $\Fac U\subseteq \T \subseteq \lperp{(\tau U)}$
  and torsion classes in $\mod C$, where
  $F$ is the equivalence obtained in Theorem \ref{ithm:U-modC}.
\end{theorem}

We would like to point out that support $\tau$-tilting modules are
related with important classes of objects in representation
theory: silting objects in triangulated categories and
cluster-tilting objects in 2-Calabi-Yau triangulated categories.
On one hand, if $\T$ is a triangulated category satisfying
suitable finiteness conditions with a silting object $S$, then
there is a bijection between basic silting objects contained in
the subcategory $S\ast S[1]$ of $\T$ and basic support
$\tau$-tilting $\End_\T(S)$-modules, see
\cite[Thm. 3.2]{adachi_tau-tilting_2012} for a special case.
On the other hand, if $\C$ is a 2-Calabi-Yau triangulated
category with a cluster-tilting object $T$, then there is a
bijection between basic cluster-tilting objects in $\C$ and
basic support $\tau$-tilting $\End_\C(T)$-modules,
see \cite[Thm. 4.1]{adachi_tau-tilting_2012}.

Reduction techniques exist both for silting objects and
cluster-tilting objects, see \cite[Thm. 2.37]{aihara_silting_2012}
and \cite[Thm. 4.9]{iyama_mutation_2008} respectively. 
The following result shows that $\tau$-tilting reduction fits
nicely in these contexts.

\begin{theorem}[see Theorems \ref{thm:silting-compatibility} and
  \ref{thm:2-cy-compatibility} for details]
  Let $A$ be a finite dimensional algebra.
  Then we have the following:
  \begin{enumerate}
  \item $\tau$-tilting reduction is compatible with silting
    reduction.
  \item If $A$ is 2-Calabi-Yau tilted, then $\tau$-tilting
    reduction is compatible with 2-Calabi-Yau reduction.
  \end{enumerate}
\end{theorem}

These results enhance our understanding of the relationship
between silting objects, cluster-tilting objects and support
$\tau$-tilting modules. 
We refer the reader to \cite{brustle_ordered_2013} for an in-depth
survey of the relations between these objects and several other
important concepts in representation theory.

Finally, let us fix our conventions and notations, which we kindly
ask the reader to keep in mind for the remainder of this article.

\begin{conventions} 
  In what follows, $A$ always denotes a (fixed) finite
  dimensional algebra over a field $k$. 
  We denote by $\mod A$ the category of finite dimensional right
  $A$-modules.
  Whenever we consider a subcategory of $\mod A$ we assume that it
  is full and closed under isomorphisms. 
  If $M$ is an $A$-module, we denote by $\Fac M$ the subcategory
  of $\mod A$ which consists of all factor modules of direct sums
  of copies of $M$; the subcategory $\Sub M$ is defined dually. 
  Given morphisms $f:X\to Y$ and $g:Y\to Z$ in some category $\C$,
  we denote their composition by $g\circ f = gf$.
  Given a subcategory $\X$ of an additive category $\C$, we
  denote by $\lperp{\X}$ the subcategory of $\C$ whose objects are
  all objects $M$ in $\C$ such that $\Hom_\C(M,\X)=0$; the
  category $\X^\perp$ is defined dually.
  Also, we denote by $[\X]$ the ideal of $\C$ of morphisms which
  factor through $\X$.  
  For an object $X$ of $\C$, we denote by $\add X$ the smallest
  additive subcategory of $\C$ containing $X$ and closed under
  isomorphisms. 
  If $\X=\add X$ for some object $X$ in $\C$ we write
  $\lperp{X}$ instead of $\lperp{\X}$ and so on. 
  If $\C$ is a $k$-linear category we denote by $D$ the usual
  $k$-duality $\Hom_k(-,k)$.
\end{conventions}

%%% Local Variables: 
%%% mode: latex
%%% TeX-master: "../master"
%%% End: 

\section{Preliminaries}
\label{sec:preliminaries}

There is a strong interplay between the classical concept of
torsion class in $\mod A$ and the recently investigated class of
support $\tau$-tilting modules. In this section we collect the
basic definitions and main results relating this two
theories.

\subsection{Torsion pairs}

Recall that a subcategory $\X$ of an additive category $\C$ is
said to be \emph{contravariantly finite in $\C$} if for every
object $M$ of $\C$ there exist some $X$ in $\X$ and a morphism
$f:X\to M$ such that for every $X'$ in $\X$ the sequence
\[
\Hom_\C(X',X) \xto{f\cdot} \Hom_\C(X',M) \to 0
\]
is exact. In this case $f$ is called a \emph{right
  $\X$-approximation}. 
Dually we define \emph{covariantly finite subcategories in $\C$}
and \emph{left $\X$-approximations}.
Furthermore, a subcategory of $\C$ is said to be
\emph{functorially finite in $\C$} if it is both 
contravariantly and covariantly finite in $\C$. 

A subcategory $\T$ of $\mod A$ is called a \emph{torsion class} if
it is closed under extensions and factor modules in $\mod
A$. Dually, \emph{torsion-free classes} are defined.
% For any subcategory $\X$ of $\mod A$,  $\X^\perp$ is a
% \emph{torsion-free class}, \ie $\X^\perp$ is closed under
% extensions and submodules in $\mod A$.
An $A$-module $M$ in $\T$ is said to be \emph{$\Ext$-projective in
  $\T$} if $\Ext^1_A(M,\T)=0$. 
If $\T$ is functorially finite in $\mod A$, then there are only
finitely many indecomposable $\Ext$-projective modules in $\T$ up
to isomorphism, and we denote by $P(\T)$ the direct sum of each
one of them.
For convenience, we will denote the set of all torsion classes in
$\mod A$ by $\tors A$, and by $\ftors A$ the subset of $\tors A$
consisting of all torsion classes which are functorially finite in
$\mod A$.

A pair $(\T,\F)$ of subcategories of $\mod A$ is called a
\emph{torsion pair} if $\F=\T^\perp$ and $\T=\lperp{\F}$. In such
case $\T$ is a torsion class and $\F$ is a torsion-free class in
$\mod A$. 
The following proposition characterizes torsion pairs in $\mod A$
consisting of functorially finite subcategories.

\begin{proposition}
  \label{prop:torsion-pairs}
  \cite[Prop. 1.1]{adachi_tau-tilting_2012}
  Let $(\T,\F)$ be a torsion pair in $\mod A$. The following
  properties are equivalent:
  \begin{enumerate}
  \item $\F$ is functorially finite in $\mod A$ (or equivalently,
    $\F$ is contravariantly finite).
  \item $\T$ is functorially finite in $\mod A$ (or equivalently,
    $\T$ is covariantly finite).
  \item $\T = \Fac P(\T)$.
  \item $P(\T)$ is a tilting $(A/\ann \T)$-module.
  \item For every $M$ in $\T$ there exists a short exact
    sequence 
    $0\to L \to T' \xto{f} M \to 0$
    where $f$ is right $(\add P(\T))$-approximation and $L$ is in
    $\T$. 
  \end{enumerate}
\end{proposition}

A torsion pair in $\mod A$ which has any of the equivalent
properties of Proposition \ref{prop:torsion-pairs} is called a
\emph{functorially finite torsion pair}. In view of property (c),
we call the $A$-module $P(\T)$ the \emph{$\Ext$-progenerator} of
$\T$.

\subsection{$\tau$-tilting theory}

Now we recall the definition of support $\tau$-tilting modules and
the results relating such modules with functorially finite torsion
classes in $\mod A$.

\begin{definition}
  \cite[Def. 0.1(a)]{adachi_tau-tilting_2012}
  Let $A$ be a finite dimensional algebra. An $A$-module $M$ is
  said to be \emph{$\tau$-rigid} if $\Hom_A(M,\tau M)=0$ where
  $\tau$ is the Auslander-Reiten translation.
\end{definition}

\begin{remark}
  \label{rmk:tau-ext}
  By the Auslander-Reiten duality formula
  \cite[Thm. IV.2.13]{assem_elements_2006}, for every $A$-module $M$ we
  have an isomorphism $D\isHom_A(M,\tau M)\cong \Ext_A^1(M,M)$. 
  Thus $M$ is rigid (\ie $\Ext_A^1(M,M)=0$) provided $M$ is
  $\tau$-rigid.
\end{remark}

The following classical result of Auslander and Smal\o\
characterizes $\tau$-rigid modules in terms of torsion
classes.

\begin{proposition}
  \cite[5.8]{auslander_almost_1981}
  \label{prop:AS}
  Let $M$ and $N$ be two $A$-modules. Then the following holds:
  \begin{enumerate}
  \item $\Hom_A(N, \tau M)=0$ if and only if $\Ext^1_A(M,\Fac
    N)=0$. 
  \item $M$ is $\tau$-rigid if and only if $M$ is
    $\Ext$-projective in $\Fac M$.
  \item  $\Fac M$ is a functorially finite torsion class in
    $\mod A$.
  \end{enumerate}
\end{proposition}

For an $A$-module $M$ and an ideal $I$ of $A$ contained in $\ann
M$, the following proposition describes the relationship between
$M$ being $\tau$-rigid as $A$-module and $\tau$-rigid as
$(A/I)$-module. 
We denote by $\tau_{A/I}$ the Auslander-Reiten translation in
$\mod(A/I)$.

\begin{proposition}
  \label{prop:tau-rigid-ideals}
  \cite[Lemma 2.1]{adachi_tau-tilting_2012}
    Let $I$ be an ideal of $A$ and $M$ and $N$ two $(A/I)$-modules.
    Then we have the following:
    \begin{enumerate}
    \item If $\Hom_A(M,\tau N)=0$, then $\Hom_{A/I}(M,
      \tau_{A/I} N)=0$.
    \item If $I=\langle e\rangle$ for some idempotent $e\in A$,
      then $\Hom_A(M,\tau N)=0$ if and only if $\Hom_{A/I}(M, 
      \tau_{A/I} N)=0$.
    \end{enumerate}
\end{proposition}

The following lemma, which is an analog of Wakamatsu's Lemma, \cf
\cite[Lemma 1.3]{auslander_applications_1991}, often comes handy.

\begin{lemma}
  \cite[Lemma 2.5]{adachi_tau-tilting_2012}
  \label{lemma:wakamatsu}
  Let $0\to L\to M\xto{f} N$ be an exact sequence. 
  If $f$ is a right $(\add M)$-approximation of $N$ and $M$ is
  $\tau$-rigid, then $L$ is in $\lperp{(\tau M)}$.
\end{lemma}

We denote the number of pairwise non-isomorphic indecomposable
summands of an $A$-module $M$ by $|M|$. 
Thus $|A|$ equals the
rank of the Grothendieck group of $\mod A$.

\begin{definition}
  \label{def:sttilt}
  \cite[Defs. 0.1(b), 0.3]{adachi_tau-tilting_2012}
  Let $M$ be a $\tau$-rigid $A$-module.
  We say that $M$ is a \emph{$\tau$-tilting $A$-module} if
  $|M|=|A|$.
  More generally, we say that $M$ is a \emph{support
    $\tau$-tilting $A$-module} if there exists an idempotent 
  $e\in A$ such that $M$ is a $\tau$-tilting 
  $(A/\langle e\rangle)$-module.
  Support tilting $A$-modules are defined analogously, 
  see \cite{ingalls_noncrossing_2009}.
  \end{definition}

\begin{remark}
  \label{rmk:trival-sttilt}
    Note that the zero-module is a support $\tau$-tilting module
    (take $e=1_A$ in Definition \ref{def:sttilt}).
    Thus every non-zero finite dimensional algebra $A$ admits at
    least two support $\tau$-tilting $A$-modules: $0$ and $A$. 
\end{remark}

The following observation follows immediately from the
Auslander-Reiten formulas and Definition \ref{def:sttilt}.

\begin{proposition}
  \label{prop:sttilt-tilt-hered}
  \cite{adachi_tau-tilting_2012}
  Let $A$ be a hereditary algebra and $M$ an $A$-module.
  Then $M$ is a $\tau$-rigid (resp. $\tau$-tilting) $A$-module if and only if
  $M$ is a rigid (resp. tilting) $A$-module.
\end{proposition}

We also need the following result:

\begin{proposition}
  \label{prop:2.2}
  \cite[Prop. 2.2]{adachi_tau-tilting_2012}
  Let $A$ be a finite dimensional algebra. The following statements hold:
  \begin{enumerate}
  \item $\tau$-tilting $A$-modules are precisely sincere support
    $\tau$-tilting $A$-modules.
  \item Tilting $A$-modules are precisely faithful support
    $\tau$-tilting $A$-modules.
  \item Any $\tau$-tilting (resp. $\tau$-rigid) $A$-module $M$ is
    a tilting (resp. partial tilting) $(A/\ann T)$-module.
  \end{enumerate}
\end{proposition}

The following result provides the conceptual framework for the
main results of this article.
It says that basic support $\tau$-tilting $A$-modules are
precisely the $\Ext$-progenerators of functorially finite torsion
classes in $\mod A$.

\begin{theorem}
  \cite[Thm. 2.2]{adachi_tau-tilting_2012}
  \label{thm:tors-sttilt}
  There is a bijection
  \[
  \ftors A \longrightarrow \sttilt A
  \]
  given by $\T\mapsto P(\T)$ with inverse $M\mapsto \Fac M$.
\end{theorem}

\begin{remark}
  \label{rmk:partial-order-sttiltA}
  Observe that the inclusion of subcategories gives a partial
  order in $\tors A$. 
  Thus the bijection of Theorem \ref{thm:tors-sttilt} induces a
  partial order in $\sttilt A$.
  Namely, if $M$ and $N$ are support $\tau$-tilting $A$-modules,
  then
  \[
  M\leq N\quad\text{if and only if}\quad\Fac M\subseteq \Fac N.
  \]
  Hence, as with every partially ordered set, we can associate to
  $\sttilt A$ a \emph{Hasse quiver} $Q(\sttilt A)$ whose set of
  vertices is $\sttilt A$ and there is an arrow $M\to N$ if and
  only if $M>N$ and there is no $L\in\sttilt A$ such that
  $M>L>N$.
\end{remark}

The following proposition is a generalization of Bongartz
completion of tilting modules, see \cite[Lemma
VI.2.4]{assem_elements_2006}.
It plays an important role in the sequel.

\begin{proposition}
  \label{prop:bongartz-completion}
  \cite[Prop. 2.9]{adachi_tau-tilting_2012}
  Let $U$ be a $\tau$-rigid $A$-module. 
  Then the following holds:
  \begin{enumerate}
  \item $\lperp{(\tau U)}$ is a functorially finite torsion class
    which contains $U$.
  \item $U$ is $\Ext$-projective in $\lperp{(\tau U)}$, that is
    $U\in \add P(\lperp{(\tau U)})$.
  \item $T_U:=P(\lperp{(\tau U)})$ is a $\tau$-tilting $A$-module.
  \end{enumerate}
  The module $T_U$ is called the \emph{Bongartz completion of $U$
    in $\mod A$.}
\end{proposition}

Recall that, by definition, a partial-tilting $A$-module $T$ is a
tilting $A$-module if and only if there exists a short exact
sequence $0\to A\to T'\to T''\to 0$ with $T',T''\in\add T$. 
The following proposition gives a similar criterion for
a $\tau$-rigid $A$-module to be a support $\tau$-tilting
$A$-module.

\begin{proposition}
  \label{prop:sttilt-sequence}
  Let $M$ be a $\tau$-rigid $A$-module. 
  Then $M$ is a support $\tau$-tilting $A$-module if and only if
  there exists an exact sequence 
  \begin{equation}
    \label{eq:T3}
    A \xto{f} M' \xto{g} M''\to 0    
  \end{equation}
  with $M',M''\in \add M$ and $f$ a left $(\add M)$-approximation
  of $A$.
\end{proposition}
\begin{proof}
  The necessity is shown in
  \cite[Prop. 2.22]{adachi_tau-tilting_2012}.
  For the sufficiency, suppose there exists an exact sequence of
  the form \eqref{eq:T3}.
  Let $I=\ann M$, we only need to find an idempotent $e\in A$
  such that $e\in I$ and $|M|=|A/\langle e\rangle|$. 
  By Proposition \ref{prop:2.2}(c) $M$ is a
  partial-tilting $(A/I)$-module.
  % Since $M$ is $\tau$-rigid, by Proposition
  % \ref{prop:AS} we have that $M$ is $\Ext$-projective in 
  % $\Fac M$. 
  % Hence $M$ is a direct summand of the tilting $(A/I)$-module
  % $P(\Fac M)$, see Proposition \ref{prop:torsion-pairs}. 
  % In particular, $M$ is a partial-tilting $(A/I)$-module. 
  Moreover, $f$ induces a morphism $\bar{f}:A/I\to M'$.
  We claim that the sequence
  \[
  0\to A/I \xto{\bar{f}} M'\xto{g} M'' \to 0
  \]
  is exact, for which we only need to show that the induced
  morphism $\bar{f}$ is injective. 
  It is easy to see that $\bar{f}:A/I\to M'$ is a left $(\add
  M)$-approximation of $A/I$. 
  Since $M$ is a faithful $(A/I)$-module, by
  \cite[VI.2.2]{assem_elements_2006} we have that $\bar{f}$ is
  injective, and the claim follows.
  Thus $M$ is a tilting $(A/I)$-module, and we have $|M|=|A/I|$. 
  Let $e$ be a maximal idempotent in $A$ such that $e\in I$.
  Then by the choice of $e$ we have that 
  $|M|=|A/I|=|A/\langle e\rangle|$.
  % and thus $|M|=|A/\langle e\rangle|$.
  % This shows that $M$ and $e$ satisfy condition (1) in Definition
  % \ref{def:sttilt}, hence $M$ is a support $\tau$-tilting
  % $A$-module. 
\end{proof}

The following result justifies the claim that support
$\tau$-tilting modules complete the class of tilting modules from
the point of view of mutation. 
We say that a basic $\tau$-rigid $A$-module $U$ is
\emph{almost-complete} if $|U|=|A|-1$. 

\begin{theorem}
  \label{thm:non-branching}
  \cite[Thm. 2.17]{adachi_tau-tilting_2012}  
  Let $U$ be an almost-complete $\tau$-tilting $A$-module.
  Then there exist exactly two basic support $\tau$-tilting
  $A$-modules having $U$ as a direct summand.
\end{theorem}

\begin{definition}
  \label{def:non-branching}
  It follows from Theorem \ref{thm:non-branching} that we can
  associate with $\sttilt A$ an \emph{exchange graph} whose
  vertices are basic support $\tau$-tilting $A$-modules and there
  is an edge between two non-isomorphic support $\tau$-tilting
  $A$-modules $M$ and $N$ if and only if the following
  holds:
  \begin{itemize}
  \item There exists an idempotent $e\in A$ such that
    $M,N\in\mod(A/\langle e\rangle)$.
  \item There exists an almost-complete $\tau$-tilting
    $(A/\langle e\rangle)$-module $U$ such that $U\in\add M$ and
    $U\in\add N$.
  \end{itemize}
  In this case we say that $M$ and $N$ are obtained from each
  other by mutation.
  Note that this exchange graph is $n$-regular, where $|A|=n$ is
  the number of simple $A$-modules.
  It is shown in \cite[Cor. 2.31]{adachi_tau-tilting_2012} that
  the underlying graph of $Q(\sttilt A)$ coincides with the
  exchange graph of $\sttilt A$.
\end{definition}

We conclude this section with some examples of support
$\tau$-tiling modules.

\begin{example}
  \label{ex:A2}
  Let $A$ be a hereditary algebra.
  By Proposition \ref{prop:2.2} support
  $\tau$-tilting $A$-modules are precisely support tilting
  $A$-modules.
  For example, let $A$ be the path algebra of the quiver $2\la1$. 
  The Auslander-Reiten quiver of $\mod A$ is given by
  
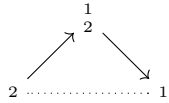
\begin{center}
    \begin{tikzpicture}[commutative diagrams/every diagram]
      \node (S2) at (mesh cs:u=0,v=0,r=0.7){$\rep{2}$};
      \node (P1) at (mesh cs:u=0,v=1,r=0.7){$\rep{1\\2}$};
      \node (S1) at (mesh cs:u=1,v=1,r=0.7){$\rep{1}$};
      \path[commutative diagrams/.cd, every arrow]
      (S2) edge (P1)
      (P1) edge (S1)
      (S1) edge[commutative diagrams/path,dotted] (S2);
    \end{tikzpicture}        
  \end{center}
  where modules are represented by their radical filtration. 
  Then $Q(\sttilt A)$ is given by
  
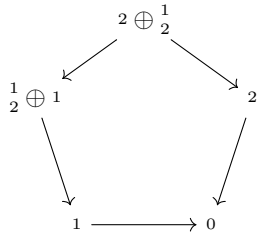
\begin{center}
    \begin{tikzpicture}[commutative diagrams/every diagram]
      \node (a) at (90:1.5)      {$\rep{2}\oplus\rep{1\\2}$};
      \node (b) at (90+72:1.5)   {$\rep{1\\2}\oplus\rep{1}$};
      \node (c) at (90+72*2:1.5) {$\rep{1}$};
      \node (d) at (90+72*3:1.5) {$\rep{0}$};
      \node (e) at (90+72*4:1.5) {$\rep{2}$};
      \path[commutative diagrams/.cd, every arrow]
      (a) edge (b)
      (b) edge (c)
      (c) edge (d)
      (a) edge (e)
      (e) edge (d);
    \end{tikzpicture}    
  \end{center}
  Note that the only $\tau$-tilting $A$-modules are
  $\rep{2}\oplus\rep{1\\2}$ and $\rep{1\\2}\oplus\rep{1}$, and
  since $A$ is hereditary they are also tilting $A$-modules.
\end{example}

\begin{example}
  \label{ex:two-cycle}
  Let $A$ be the algebra given by the quiver
  
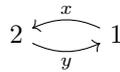
\begin{center}
    \begin{tikzcd}
      2\rar[bend right=25,swap]{y} & 1 \lar[bend right=25,swap]{x}
    \end{tikzcd}    
  \end{center}
  subject to the relation $yx=0$.
  The Auslander-Reiten quiver of $\mod A$ is given by
  
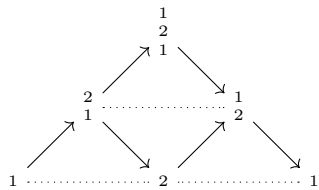
\begin{center}
    \begin{tikzpicture}
      \node (S1) at (mesh cs:u=0,v=0,r=0.7) {$\rep{1}$};
      \node (P2) at (mesh cs:u=0,v=1,r=0.7) {$\rep{2\\1}$};
      \node (P1) at (mesh cs:u=0,v=2,r=0.7) {$\rep{1\\2\\1}$};
      \node (I2) at (mesh cs:u=1,v=2,r=0.7) {$\rep{1\\2}$};
      \node (S2) at (mesh cs:u=1,v=1,r=0.7) {$\rep{2}$};
      \node (S1a) at (mesh cs:u=2,v=2,r=0.7) {$\rep{1}$};
      \path[commutative diagrams/.cd, every arrow]
      (S1) edge (P2) 
      (P2) edge (P1)
      (P2) edge (S2)
      (P1) edge (I2)
      (S2) edge (I2)
      (I2) edge (S1a)
      
      (S1) edge[commutative diagrams/path,dotted] (S2)
      (S1a) edge[commutative diagrams/path,dotted] (S2)
      (P2) edge[commutative diagrams/path,dotted] (I2)      
      ;
    \end{tikzpicture}
  \end{center}
  where the two copies of $S_1=1$ are to be identified.
  Then $Q(\sttilt A)$ is given as follows:
  
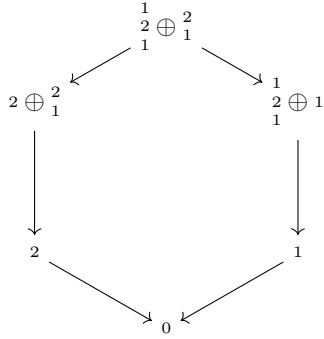
\begin{center}    
    \begin{tikzpicture}
      \node (a) at (90:2)    {$\rep{1\\2\\1}\oplus\rep{2\\1}$};
      \node (b) at (90+60:2) {$\rep{2}\oplus\rep{2\\1}$};
      \node (c) at (90+60*2:2) {$\rep{2}$};
      \node (d) at (90+60*3:2) {$\rep{0}$};
      \node (e) at (90+60*4:2) {$\rep{1}$};
      \node (f) at (90+60*5:2) {$\rep{1\\2\\1}\oplus\rep{1}$};
      \path[commutative diagrams/.cd, every arrow]
      (a) edge (b)
      (b) edge (c)
      (c) edge (d)
      (a) edge (f)
      (f) edge (e)
      (e) edge (d);
    \end{tikzpicture}
  \end{center}
\end{example}

\begin{example}
  \label{ex:Nak3}
  Let $A$ be a self-injective algebra. Then the only basic tilting
  $A$-module is $A$. On the other hand, in general there are many
  basic support $\tau$-tilting modules.
  For example, let $A$ be the path algebra of the quiver
  
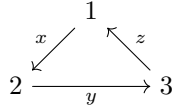
\begin{center}
    \begin{tikzpicture}[commutative diagrams/every diagram]
      \node (S2) at (mesh cs:u=0,v=0,r=0.7){$2$};
      \node (P1) at (mesh cs:u=0,v=1,r=0.7){$1$};
      \node (S1) at (mesh cs:u=1,v=1,r=0.7){$3$};
      \path[commutative diagrams/.cd, every arrow, every
      label]
      (P1) edge node [swap]{$x$} (S2)
      (S2) edge node [swap]{$y$} (S1)
      (S1) edge node [swap]{$z$} (P1);
    \end{tikzpicture}        
  \end{center}
  subject to the relations $xy=0$, $yz=0$ and $zx=0$. 
  Thus $A$ is a self-injective cluster-tilted algebra of type
  $A_3$,
  see \cite{buan_cluster-tilted_2007,ringel_self-injective_2008}.  
  It follows from \cite[Thm. 4.1]{adachi_tau-tilting_2012} that
  basic support $\tau$-tilting $A$-modules correspond bijectively
  with basic cluster-tilting objects in the cluster category of
  type $A_3$. 
  Hence there are 14 support $\tau$-tilting $A$-modules, 
  see \cite[Fig. 4]{buan_tilting_2006}. 
\end{example}

The following example gives an algebra with infinitely many
support $\tau$-tilting modules.

\begin{example}
  \label{ex:kronecker}
  Let $A$ be the Kronecker algebra, \ie the path algebra of the
  quiver $2\leftleftarrows 1$.
  Then $Q(\sttilt A)$ is the following quiver,
  where each module is represented by its radical filtration:
  
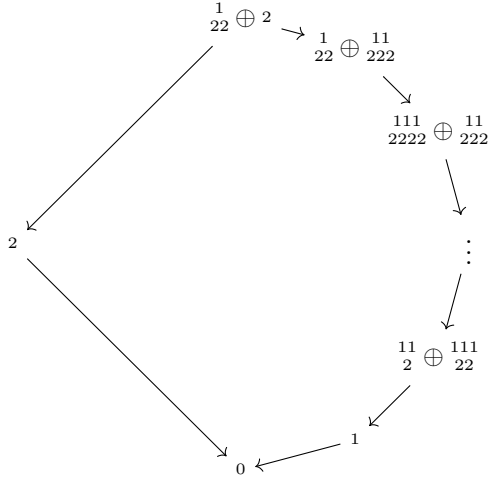
\begin{center}
    \begin{tikzpicture}
  \node (1) at (90:3)    {$\rep{1\\22}\oplus\rep{2}$};
  \node (2) at (90-30:3) {$\rep{1\\22}\oplus\rep{11\\222}$};
  \node (3) at (90-30*2:3) {$\rep{111\\2222}\oplus\rep{11\\222}$};
  \node (4) at (90-30*3:3) {$\vdots$};
  \node (5) at (90-30*4:3) {$\rep{11\\2}\oplus\rep{111\\22}$};
  \node (6) at (90-30*5:3) {$\rep{1}$};
  \node (7) at (90-30*6:3) {$\rep{0}$};
  \node (8) at (90+30*3:3) {$\rep{2}$};
  \path[commutative diagrams/.cd, every arrow]
  (1) edge (2)
  (2) edge (3)
  (3) edge (4)
  (4) edge (5)
  (5) edge (6)
  (6) edge (7)
  (1) edge (8)
  (8) edge (7);
\end{tikzpicture}

%%% Local Variables: 
%%% mode: latex
%%% TeX-master: "../master"
%%% End:     
  \end{center}
\end{example}

%%% Local Variables: 
%%% mode: latex
%%% TeX-master: "../master"
%%% End: 

\section{Main results}
\label{sec:reduction}

This section is devoted to prove the main results of this article.
First, let us fix the setting of our results.

\begin{setting}
  \label{set:main-results}
  We fix a finite dimensional algebra $A$ and a basic $\tau$-rigid
  $A$-module $U$. 
  Let $T=T_U$ be the Bongartz completion of $U$ in $\mod A$, see
  Proposition \ref{prop:bongartz-completion}.
  The algebras 
  \[ 
  B = B_U := \End_A(T_U)
  \quad\text{and}\quad
  C = C_U := B_U/\langle e_U \rangle
  \]
  play an important role in the sequel, where $e_U$ is the
  idempotent corresponding to the projective $B$-module
  $\Hom_A(T_U,U)$. 
  We regard $\mod C$ as a full subcategory of $\mod B$ \emph{via}
  the canonical embedding. 
\end{setting}

In this section we study the subset of $\sttilt A$ given by 
\[
\sttilt_U A := \setP{M\in\sttilt A}{U\in\add M}.
\]
In Theorem \ref{thm:tau-tilting-reduction} we will show that there
is an order-preserving bijection between $\sttilt_U A$ and
$\sttilt C$.

\subsection{The $\tau$-perpendicular category}

The following observation allows us to describe $\sttilt A$ in
terms of the partial order in $\tors A$. 

\begin{proposition}
  \cite[Prop. 2.8]{adachi_tau-tilting_2012}
  \label{prop:sttiltUA}
  Let $U$ be a $\tau$-rigid $A$-module and $M$ a support
  $\tau$-tilting $A$-module. 
  Then, $U\in\add M$ if and only if 
  \[
  \Fac U\subseteq \Fac M \subseteq \lperp{(\tau U)}.
  \]
\end{proposition}

Recall that we have $M\leqslant N$ for two basic support
$\tau$-tilting $A$-modules if and only if $\Fac M\subseteq\Fac N$,
see Remark \ref{rmk:partial-order-sttiltA}.
Hence it follows from Proposition \ref{prop:sttiltUA} that
$\sttilt_U A$ is an interval in $\sttilt A$, \ie  we have that
\begin{equation}
  \label{eq:-sttiltUA-interval}
\sttilt_U A = \setP{M\in\sttilt A}{P(\Fac U)\leqslant M\leqslant
  T_U}
\end{equation}
In particular there are two distinguished functorially finite
torsion pairs associated with $P(\Fac U)$ and $T_U$.
Namely, 
\[
(\Fac U,U^\perp)
\quad\text{and}\quad
(\lperp{(\tau U)},\Sub\tau U)
\] 
which satisfy $\Fac U\subseteq\lperp{(\tau U)}$ and 
$\Sub\tau U\subseteq U^\perp$. 

\begin{definition}
  \label{def:tau-perpendicular}
  The \emph{$\tau$-perpendicular category associated to $U$} is
  the subcategory of $\mod A$ given by
  $\U:=\lperp{(\tau U)}\cap U^\perp$.
\end{definition}

The choice of terminology in Definition
\ref{def:tau-perpendicular} is justified by the following
example.

\begin{example}
  \label{ex:tau-perpendicular}
  Suppose that $U$ is a partial-tilting $A$-module.
  Since $U$ has projective dimension less or equal than 1.
  Then, by the Auslander-Reiten formulas, for every $A$-module $M$
  we have that $\Hom_A(M,\tau U)=0$ if and only if
  $\Ext_A^1(U,M)=0$. 
  Then
  \[
  \U=\setP{M\in\mod A}{\Hom(U,M)=0\text{ and }\Ext_A^1(U,M)=0}.
  \]
  Thus $\U$ is exactly the right perpendicular category associated
  to $U$ in the sense of \cite{geigle_perpendicular_1991}.
\end{example}

We need a simple observation which is a consequence 
of a theorem of Brenner and Butler.

\begin{proposition}
  \label{prop:bb-tilting-thm}
  With the hypotheses of Setting \ref{set:main-results}, the
  functors  
    \begin{align}
      F:=\Hom_A(T,-):&\ \mod A\to \mod B
      \quad\text{and} \label{eq:F}\\ 
      G:=-\otimes_B T:&\ \mod B\to \mod A \label{eq:G}
    \end{align}  
    induce mutually quasi-inverse equivalences $F:\Fac T
    \to \Sub DT$ and $G:\Sub DT\to \Fac T$.
    Moreover, these equivalences are \emph{exact}, \ie $F$ sends
    short exact sequences in $\mod A$ with terms in $\Fac T$ to
    short exact sequences in $\mod B$, and so does $G$.
\end{proposition}
\begin{proof}
  In view of Proposition \ref{prop:2.2}, we have that
  $T$ is a tilting $(A/\ann T)$-module.
  Then it follows from \cite[Thm. VI.3.8]{assem_elements_2006}
  that $F:\Fac T\to \Sub DT$ is an equivalence with quasi-inverse 
  $G:\Sub DT\to \Fac T$.

  Now we show that both $F$ and $G$ are exact.
  For this, let $0\to L\to M\to N\to 0$ be a short exact sequence
  in $\mod A$ with terms in $\Fac T$.
  Then $F$ induces an exact sequence
  \[
  0\to F L\to F M\to F N\to \Ext_A^1(T,L)=0
  \]
  as $T$ is $\Ext$-projective in $\lperp{(\tau U)}$.
  Consequently $F$ is exact.

  Next, let 
  $0\to L'\to M'\to N'\to 0$ be a short exact sequence in $\mod A$
  with terms in $\Sub DT$, then there is an exact sequence
  \[
  0=\Tor_1^B(N',T)\to GL' \to GM' \to GN'
  \to 0.
  \]
  as $N'\in\Sub DT=\ker\Tor_1^B(-,T)$, 
  see \cite[Cor. VI.3.9(i)]{assem_elements_2006}; hence $G$ is
  also exact. 
\end{proof}

The following proposition gives us a basic property of $\U$.

\begin{proposition}
  \label{prop:2-out-of-3}
  Let $0\to L\to M\to N\to 0$ be an exact sequence in $\mod A$.
  If any two of $L,\ M$ and $N$ belong to $\U$, then the third one
  also belongs to $\U$.
\end{proposition}
\begin{proof}
  First, $\U$ is closed under extensions since both
  $\lperp{(\tau U)}$ and $U^\perp$ are closed under
  extensions in $\mod A$.
  Thus if $L$ and $N$ belong to $\U$, then so does $M$.

  Secondly, suppose that $L$ and $M$ belong to
  $\U$. Since $\lperp{(\tau U)}$ is closed under factor modules we
  only need to show that $\Hom_A(U,N)=0$. In this case we have an
  exact sequence 
  \[
  0=\Hom_A(U,M) \to \Hom_A(U,N) \to \Ext_A^1(U,L)=0
  \]
  since by Proposition \ref{prop:bongartz-completion}(b) we have
  that $U$ is $\Ext$-projective in $\lperp{(\tau U)}$, hence $N$
  is in $\U$.

  Finally, suppose that $M$ and $N$ belong to $\U$.
  Since $U^\perp$ is closed under submodules, we only need to show
  that $\Hom_A(L,\tau U)=0$. 
  We have an exact sequence
  \[
  0=\Hom_A(M,\tau U) \to \Hom_A(L,\tau U) \to \Ext_A^1(N,\tau U).
  \]
  By the dual of Proposition \ref{prop:bongartz-completion}(b) we
  have that $\tau U$ is $\Ext$-injective in $U^\perp$, so we have
  $\Ext_A^1(N,\tau U)=0$, hence $\Hom_A(L,\tau U)=0$ and thus $L$
  is in $\U$. 
\end{proof}

\begin{remark}
  \label{rmk:U-exact}
  Since $\U$ is closed under extensions in $\mod A$, it has a
  natural structure of an exact category, 
  see \cite{quillen_higher_1973,keller_derived_1996}.
  Then Proposition \ref{prop:2-out-of-3} says that admissible
  epimorphisms (resp. admissible monomorphisms) in $\U$ are exactly
  epimorphisms (resp. monomorphisms) in $\mod A$ between modules in
  $\U$.
\end{remark}

The next result is the main result of this subsection.
It is the first step towards $\tau$-tilting reduction.

\begin{theorem}
  \label{thm:U-modC}
  With the hypotheses of Setting \ref{set:main-results}, the
  functors 
  $F$ and $G$ in  Proposition \ref{prop:bb-tilting-thm} induce
  mutually quasi-inverse equivalences 
  $F:\U\to \mod C$ and $G:\mod C\to\U$ which are exact.
\end{theorem}
\begin{proof}
  By Proposition \ref{prop:bb-tilting-thm}, the functors
  \eqref{eq:F} and \eqref{eq:G} induce mutually quasi-inverse
  equivalences between $\Fac T$ and $\Sub DT$.
  Hence by construction we have $F(\U)\subseteq
  (FU)^\perp=\mod C$.
  Thus, % to show that $F:\U \to\mod C $ and $G:\mod C\to\U$ are
  % mutually quasi-inverse equivalences
  we only need to show that
  $F:\U\to \mod C$ is dense. 

  Let $N$ be in $\mod C$ and take a projective presentation 
  \begin{equation}
    \label{eq:proj-pres-N}
    F T_1 \xto{Ff} F T_0 \to N \to 0
  \end{equation}
  of the $B$-module $N$; hence we have $T_0,T_1\in T_U$. 
  Let $M = \Coker f$. We claim that $F M \cong N$ and that $M$ is
  in $\U$.
  In fact, let $L = \Im f$ and $K=\Ker f$.
  Then we have short exact sequences
  \begin{align}
    &0 \to K \to T_1 \to L \to 0
    \quad\text{and} \label{eq:T-resolution-1}\\
    &0 \to L \to T_0 \to M \to 0
    \label{eq:T-resolution-2}.
  \end{align}
  Then $L$ is in $\lperp{(\tau U)}$ since $\lperp{(\tau U)}$
  is closed under factor modules. 
  In particular we have that $\Ext^1_A(T,L)=0$.
  Apply the functor $F$ to the short exact sequences
  \eqref{eq:T-resolution-1} and \eqref{eq:T-resolution-2} to
  obtain a commutative diagram with exact rows and columns
  \begin{center}
    \begin{tikzcd}
      {}& F T_1 \dar \rar{Ff} & F T_0 \rar \dar[equals] &
      F M \dar[equals] \\ 
      0 \rar & F L \rar \dar & F T_0 \rar & F M \rar &
      \Ext_A^1(T,L)=0 \\ 
      & \Ext_A^1(T,K)
    \end{tikzcd}
  \end{center}  
  To prove that $FM\cong N$ it remains to show that
  $\Ext^1_A(T,K)=0$.
  Since $T$ is $\Ext$-projective in $\lperp{(\tau U)}$, it
  suffices to show that $K$ is in $\lperp{(\tau U)}$. 
  Applying the functor $\Hom_A(-,\tau U)$ to
  \eqref{eq:T-resolution-1}, we obtain an exact sequence 
  \[
  0=\Hom_A(T_1,\tau U) \to \Hom_A(K,\tau U) \to 
  \Ext^1_A(L,\tau U)\xto{\Ext^1_A(f,\tau U)}\Ext_A^1(T_1,\tau U).
  \]
  Thus we only need to show that the map 
  $\Ext_A^1(f,\tau U):\Ext_A^1(L,\tau U)\to \Ext_A^1(T_1,\tau U)$
  is a monomorphism.
  By Auslander-Reiten duality it suffices to show that the map
  \[
  \psHom_A(U,T_1) \xto{f\circ-} \psHom_A(U,L)
  \]
  is an epimorphism.
  For this, observe that we have a commutative diagram
  \begin{center}
    \begin{tikzcd}
      \Hom_A(U,T_1) \rar{f\circ-} \dar & \Hom_A(U,L)
      \dar \\ 
      \psHom_A(U,T_1) \rar[swap]{f\circ-} & \psHom_A(U,L)
    \end{tikzcd}    
  \end{center}
  where the vertical maps are natural epimorphisms.
  Hence it is enough to show that the map 
  \begin{equation}
    \label{eq:this-map-is-surjective}
    \Hom_A(U,T_1) \xto{f\circ-} \Hom_A(U,L)
  \end{equation}
  is surjective.
  Applying the functor $\Hom_B(FU,-)$ to the
  sequence \eqref{eq:proj-pres-N} we obtain an exact sequence
  \[
  \Hom_B(FU,FT_1)
  \xto{\Hom_B(FU,Ff)}
  \Hom_B(FU,FT_0)\to\Hom_B(FU,N)=0
  \]
  since $FU$ is a projective $B$-module and $N$ is in $\mod
  C=(FU)^\perp$.
  Thus $\Hom_B(FU,Ff)$ is surjective, and also the map
  \eqref{eq:this-map-is-surjective} is surjective by 
  Proposition \ref{prop:bb-tilting-thm}.
  Hence we have that $K$ belongs to $\lperp{(\tau U)}$ as
  desired. 
  This shows that $F M \cong N$.

  Moreover, we have that
  \[
  0=\Hom_B(FU, N) \cong \Hom_B(FU, F M) \cong\Hom_A(U, M).
  \]
  Hence $M$ is in $\U$.
  This shows that $F:\U \to \mod C$ is dense, hence 
  $F$ is an equivalence with quasi-inverse $G$. 
  The fact that this equivalences are exact follows immediately
  from Proposition \ref{prop:bb-tilting-thm}. 
  This concludes the proof of the theorem.
\end{proof}

\begin{definition}
  \label{def:torsU}
  We say that a full subcategory $\G$ of $\U$ is a \emph{torsion class}
  in  $\U$ if the following holds: 
  Let $0\to X \to Y\to Z\to 0$ be an admissible exact sequence in
  $\U$, see Remark \ref{rmk:U-exact}:
  \begin{enumerate}
  \item If $X$ and $Z$ are in $\G$, then $Y$ is in $\G$.
  \item If $Y$ is in $\G$, then $Z$ is in $\G$.
  \end{enumerate}
  We denote the set of all torsion classes in $\U$ by $\tors\U$. 
  We denote by $\ftors\U$ the subset of $\tors \U$ consisting of
  torsion classes which are functorially finite in $\U$.
\end{definition}

\begin{example}
  \label{ex:perpendicular-modC}
  If $A$ is hereditary and $U$ a basic partial-tilting
  $A$-module then by \cite{geigle_perpendicular_1991}  the algebra
  $C$ is hereditary and Theorem \ref{thm:U-modC} specializes to a
  well-known result from \opcit.
\end{example}

The following corollary is an immediate consequence of Theorem
\ref{thm:U-modC}.

\begin{corollary} 
  \label{cor:torsU}
  The following holds:
  \begin{enumerate}
  \item The functors $F$ and $G$ induce mutually inverse
    bijections between $\tors\U$ and $\tors C$.
  \item These bijections restrict to bijections between $\ftors\U$
    and $\ftors C$.
  \item These bijections above are isomorphisms of partially ordered
    sets. 
  \end{enumerate}
\end{corollary}
\begin{proof}
It is shown in Theorem \ref{thm:U-modC} that $F$ and $G$ give
equivalences of exact categories between $\U$ and $\mod C$. 
Since the notion of torsion class depends only on the exact
structure of the category, see Definition
\ref{def:torsU}, part (a) follows. 
Now (b) and (c) are clear.  
\end{proof}

\subsection{Reduction of torsion classes and $\tau$-tilting
  modules}

Given two subcategories $\X$ and $\Y$ of $\mod A$ we denote by 
$\X*\Y$ the full subcategory of $\mod A$ induced by all
$A$-modules $M$ such that there exist a short exact sequence
\[
0\to X \to M \to Y \to 0
\]
with $X$ in $\X$ and $Y$ in $\Y$. 
Obviously we have $\X,\Y\subseteq \X *\Y$. 
The following two results give us reductions at the level of
torsion pairs.

\begin{theorem}
  \label{thm:tp-reduction}
  With the hypotheses of Setting \ref{set:main-results}, we have
  order-preserving bijections
  \[
  \setP{\T\in\tors A}{\Fac U\subseteq\T\subseteq\lperp{(\tau U)}}
  \stackrel{\taured}{\longrightarrow}
   \tors\U\stackrel{F}{\longrightarrow} \tors C
  \]
  where $\taured$ is given by $\taured(\T):=\T\cap U^\perp$ with
  inverse  
  $\taured^{-1}(\G):=(\Fac U)*\G$,  and $F$ is given in
  Corollary \ref{cor:torsU}. 
\end{theorem}

The situation of Theorem \ref{thm:tp-reduction} is illustrated in
Figure \ref{fig:tp-reduction-venn}.
Also, the following diagram is helpful to visualize this reduction
procedure: 
\begin{center}
  \begin{tikzcd}[row sep=tiny, column sep=large]
    \lperp{(\tau U)} \rar[mapsto] & 
    \lperp{(\tau U)}\cap U^\perp =\U \\ 
    \rotatebox[origin=c]{90}{$\subseteq$} &
    \rotatebox[origin=c]{90}{$\subseteq$} \\
    \T \rar[mapsto] & \T\cap U^\perp \\
    \rotatebox[origin=c]{90}{$\subseteq$} &
    \rotatebox[origin=c]{90}{$\subseteq$} \\
    \Fac U \rar[mapsto] & (\Fac U)\cap U^\perp =\set{0}
  \end{tikzcd}
\end{center}

Moreover, we have the following bijections:

\begin{figure}[t]
  \centering
  \begin{tikzpicture}[commutative diagrams/every diagram]
  \begin{scope}
    \clip (-2.5,0) ellipse (1.5 and 1.5);
    \fill[black]  (-0.1,0) ellipse (1.5 and 0.9);
  \end{scope}
  \node at (-2,0) {$\U$};
  \draw (0,0) circle (0.7) node {$\Fac U$};
  \draw (-0.1,0) ellipse (1.5 and 0.9); % T
  \node at ($(-0.5,0)+(1.5,0)$) {$\T$};
  \draw (0.2,0) ellipse (2.75 and 1.5); % lperp(tU)
  \node at ($(-0.5,0)+(2.6,0)$) {$\lperp{(\tau U)}$};
  \draw (-2.5,0) ellipse (1.5 and 1.5); % U^perp
  \node at ($(-2.5,0)-(0.6,0)$) {$U^\perp$};
\end{tikzpicture}

%%% Local Variables:
%%% TeX-master: "../master"
%%% End:
  \caption{Reduction of torsion classes in $\mod A$ with respect
    to $\Fac U$, see Theorems \ref{thm:tp-reduction} and
    \ref{thm:ff-tp-reduction}.} 
  \label{fig:tp-reduction-venn} 
\end{figure}
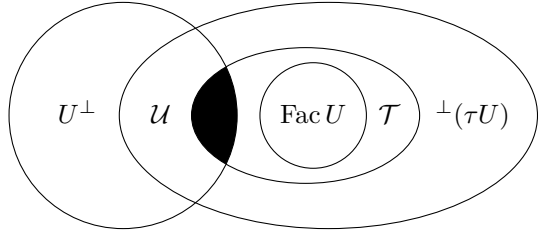  

\begin{theorem}
  \label{thm:ff-tp-reduction}
  The bijections of Theorem \ref{thm:tp-reduction} restrict to
  order-preserving bijections 
  \[
  \setP{\T\in\ftors A}{\Fac U\subseteq\T\subseteq\lperp{(\tau U)}}
  \stackrel{\taured}{\longrightarrow} 
  \ftors\U\stackrel{F}{\longrightarrow}\ftors C. 
  \]
\end{theorem}

For readability purposes, the proofs of Theorems
\ref{thm:tp-reduction} and \ref{thm:ff-tp-reduction} are given
in Section \ref{sec:proofs-of-main-thms}.
First we use them to establish the bijection between $\sttilt_U A$
and $\sttilt C$. 

Recall that the torsion pair $(\Fac U, U^\perp)$ gives functors
$\t:\mod A\to \Fac U$ and $\f:\mod A\to U^\perp$ and natural
transformations $\t\to 1_{\mod A} \to \f$ such that the sequence 
\begin{equation}
  \label{eq:canonical-sequence}
  0\to \t M\to M \to \f M \to 0  
\end{equation}
is exact for each $A$-module $M$. 
The sequence \eqref{eq:canonical-sequence} is called a
\emph{canonical sequence}, and the functor $\t$ is called the
\emph{idempotent radical} associated to the torsion pair 
$(\Fac U, U^\perp)$.
Any short exact sequence $0\to L\to M\to N\to 0$ such that $L$ is
in $\Fac U$ and $N$ is in $U^\perp$ is isomorphic to the canonical
sequence of $M$, see \cite[Prop. VI.1.5]{assem_elements_2006}.
Note that since $\lperp{(\tau U)}$ is closed under factor modules
we have
\begin{equation}
  \label{eq:f-lperp}
  \f(\lperp{(\tau U)}) \subseteq \U.
\end{equation}

\begin{proposition}
  \label{prop:technical-reasons}
  Let $\T$ be a functorially finite torsion class in $\mod A$. 
  Then $\f P(\T)$ is $\Ext$-projective in $\T\cap U^\perp$ and for
  every $A$-module $N$  which is $\Ext$-projective in $\T\cap U^\perp$ we have
  $M\in \add(\f P(\T))$.
\end{proposition}
\begin{proof}
  Applying the functor $\Hom_A(-,N)$ to
  \eqref{eq:canonical-sequence}, for any $N$ in $\T\cap U^\perp$,
  we have an exact sequence 
  \[
  0=\Hom_A(\t P(\T),N) \to \Ext_A^1(\f P(\T),N) 
  \to \Ext_A^1(P(\T),N)= 0
  \]
  This shows that $\f P(\T)$ is $\Ext$-projective in  
  $\T\cap U^\perp$. 

  Now let $N$ be $\Ext$-projective in $\T\cap U^\perp$. 
  Since $N\in \T$, by Proposition \ref{prop:torsion-pairs}(e) there
  exist a short exact sequence $0\to L\to M\to N\to 0$ with
  $M\in\add(P(\T))$ and $L\in \T$.
  Since $\f N=N$ as $N\in U^\perp$, by the functoriality of $\f$
  we have a commutative diagram with exact rows:
  \begin{center}
    \begin{tikzcd}
      0\rar & L \rar\dar & M \rar{f}\dar & N \dar[equals]\rar & 0 \\
      0 \rar & K:=\Ker \f f \rar & \f M \rar{\f f} & \f N=N \rar & 0
    \end{tikzcd}
  \end{center}
  As the map $M\to\f M$ is surjective and the map $N\to\f N$
  is bijective, by the snake lemma we have that the map $L\to K$
  is surjective.
  Thus, since $L\in\T$, we have that $K$ also belongs to $\T$.
  Moreover, $K$ is a submodule of $\f M\in U^\perp$, hence $K$ is
  also in $U^\perp$.
  Since $N$ is $\Ext$-projective in $\T\cap U^\perp$ and we have
  $K\in\T\cap U^\perp$, the lower sequence splits.
  Thus $N\in\add(\f P(\T))$.
\end{proof}

We are ready to state the main result of this article, which gives
the procedure for $\tau$-tilting reduction.

\begin{theorem}
  \label{thm:tau-tilting-reduction}
  With the hypotheses of Setting \ref{set:main-results}, we have
  an order-preserving bijection 
  \[
  \taured:\sttilt_U A \longrightarrow \sttilt C
  \]
  given by $M\mapsto F(\f M)$ with inverse 
  $N\mapsto P({(\Fac U)*G(\Fac N)})$. 
  In particular, $\sttilt C$ can be embedded as an interval in
  $\sttilt A$.
\end{theorem}
\begin{proof}
  By Theorems \ref{thm:tors-sttilt} and \ref{thm:ff-tp-reduction}
  we have a commutative diagram
  \begin{center}
    \begin{tikzcd}
      \setP{\T\in\ftors A}{\Fac U\subseteq\T\subseteq\lperp{(\tau
          U)}} \rar & \ftors C \dar\\
      \sttilt_U A \uar \rar[dashed] & \sttilt C
    \end{tikzcd}
  \end{center}
  in which arrow is a bijection, and where the dashed arrow is
  given by $M\mapsto P((\Fac M)\cap U^\perp)$ and the inverse is
  given by $N\mapsto P((\Fac U)* G(\Fac N))$. 
  Hence to prove the theorem we only need to show that for any
  $M\in\sttilt_U A$ we have $P(F(\Fac M\cap U^\perp))=F(\f M)$.
  Indeed, it follows from Proposition \ref{prop:technical-reasons}
  that $\f M$ is the $\Ext$-progenerator of 
  $(\Fac M)\cap U^\perp$; and since $F$ is an exact equivalence, 
  see Theorem \ref{thm:U-modC}, we have that $F(\f M)$ is the
  $\Ext$-progenerator of $F(\Fac M\cap U^\perp)$ which is exactly
  what we needed to show.
\end{proof}

\begin{corollary}
  \label{cor:mutations}
  The bijection in Theorem \ref{thm:tau-tilting-reduction} is
  compatible with mutation of support $\tau$-tilting modules.
\end{corollary}
\begin{proof}
  It is shown in \cite[Cor. 2.31]{adachi_tau-tilting_2012} that
  the exchange graph of $\sttilt A$ coincides with the  Hasse
  diagram of $\sttilt A$.
  Since $\tau$-tilting reduction preserves the partial order in
  $\ftors$ (and hence in $\sttilt$) the claim follows.
\end{proof}

As a special case of Theorem \ref{thm:tau-tilting-reduction} we
obtain an independent proof of
\cite[Thm. 2.17]{adachi_tau-tilting_2012}.

\begin{corollary}
  \label{cor:2-complements-coro}
  Every almost-complete support $\tau$-tilting $A$-module $U$ is
  the direct summand of exactly two support $\tau$-tilting
  $A$-modules: $P(\Fac U)$ and $T_U$. 
\end{corollary}
\begin{proof}
  Let $U$ be an almost complete support $\tau$-tilting
  $A$-module.
  Clearly, we have $\set{P(\Fac U),T_U}\subseteq \sttilt_U A$.
  On the other hand, $|U|=|A|-1=|T_U|-1$ and thus $|C|=1$, see
  Setting \ref{set:main-results}. 
  By Theorem \ref{thm:tau-tilting-reduction} we have a bijection
  between $\sttilt_U A$ and $\sttilt C$, and since $|C|$=1 we have
  that $\sttilt C = \set{0,C}$,
  see \cite[Ex. 6.1]{adachi_tau-tilting_2012}.
  Thus $|\sttilt_U A|=|\sttilt C|=2$ and we have the assertion.
\end{proof}

For a finite dimensional algebra $A$ let $\stilt A$ be the set of
(isomorphism classes  of) basic support tilting $A$-modules and,
if $U$ is a partial-tilting $A$-module, let $\stilt_U A$ be the
subset of $\stilt A$ defined by
\[
\stilt_U A:=\setP{M\in\stilt A}{U\in\add M}.
\]
If we restrict ourselves to hereditary algebras we obtain the
following improvement of \cite[Thm. 3.2]{happel_links_2010}.

\begin{corollary}
  \label{cor:tilting-reduction}
  Let $A$ be a hereditary algebra. With the hypotheses of Setting
  \ref{set:main-results}, we have the following:
  \begin{enumerate}
    \item The algebra $C$ is hereditary.
    \item There is an order-preserving bijection 
  \[
  \taured:\stilt_U A \longrightarrow \stilt C
  \]
  given by $M\mapsto F(\f M)$ with inverse 
  $N\mapsto P({(\Fac U)*G(\Fac N)})$.
  \end{enumerate}
\end{corollary}
\begin{proof}
  Since $A$ is hereditary, the $\tau$-rigid module $U$ is a
  partial-tilting module. 
  Moreover, as explained in Example \ref{ex:perpendicular-modC} we
  have that $C$ is also a hereditary algebra.
  Then by Proposition \ref{prop:sttilt-tilt-hered} we have
  $\sttilt_U A=\stilt_U A$ and $\sttilt C = \stilt C$.
  Then the claim follows from Theorem
 \ref{thm:tau-tilting-reduction}.
\end{proof}

We conclude this section with some examples illustrating our
results. 

\begin{example}
  \label{ex:A3r2-catU}
  Let $A$ be the algebra given by the quiver $3\xla{y} 2\xla{x} 1$
  with the relation $xy=0$, 
  see \cite[Ex. 6.4]{adachi_tau-tilting_2012}. 
  Consider the support $\tau$-tilting $A$-module
  $U=P_3$. 
  Then
  \[
  \U=\lperp{(\tau U)}\cap U^\perp=(\mod A)\cap U^\perp=U^\perp.
  \]
  Moreover, the Bongartz completion of $U$ is given by
  $T=P_1\oplus P_2\oplus P_3$, which is the basic
  progenerator of $\lperp{(\tau U)}=\mod A$.
  Hence $C=\End_A(P_1\oplus P_2)\cong
  k(\bullet \la \bullet)$, see Example \ref{ex:A2}.
  We may visualize this in the Auslander-Reiten quiver of $\mod
  A$, where each $A$-module is represented by its radical
  filtration: 
  
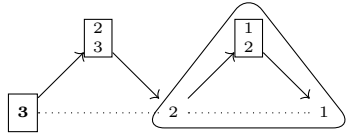
\begin{center}
    \begin{tikzpicture}[commutative diagrams/every diagram]
  \node[draw,rectangle] (S1) at (mesh cs:u=0,v=0,r=0.7){$\rep{\mathbf{3}}$};
E  \node[draw,rectangle] (P2) at (mesh cs:u=0,v=1,r=0.7){$\rep{2\\3}$};
  \node                 (S2) at (mesh cs:u=1,v=1,r=0.7){$\rep{2}$};
  \node[draw,rectangle] (P3) at (mesh cs:u=1,v=2,r=0.7){$\rep{1\\2}$};
  \node                 (S3) at (mesh cs:u=2,v=2,r=0.7){$\rep{1}$};
  \draw[rounded corners=8pt] ($(S2)-(0.4,0.2)$) -- ($(P3)+(0,0.6)$) --
  ($(S3)-(-0.4,0.2)$) -- cycle;
  \path[commutative diagrams/.cd, every arrow]
  (S1) edge (P2)
  (P2) edge (S2)
  (S2) edge (P3)
  (P3) edge (S3)
  (S1) edge[commutative diagrams/path, dotted] (S2)
  (S2) edge[commutative diagrams/path, dotted] (S3);
\end{tikzpicture}

%%% Local Variables: 
%%% mode: latex
%%% TeX-master: "../master"
%%% End: 
  \end{center}
  The indecomposable summands of $T$ are indicated with rectangles
  and $\U$ is enclosed in a triangle. 
  Note that $\U$ is equivalent
  to $\mod C$ as shown in Theorem \ref{thm:U-modC}. 

  By Theorem \ref{thm:tau-tilting-reduction} we have that 
  $\sttilt C$ can be embedded as an interval in $\sttilt A$.
  We have indicated this embedding in $Q(\sttilt A)$ in Figure
  \ref{fig:exchange-graph-sttilt-A3r2} by enclosing the image of
  $\sttilt C$ in rectangles.
  \begin{figure}[t]
    \centering
    \scalebox{0.75}{\begin{tikzpicture}[commutative diagrams/every diagram]
  \node[draw,shape=rectangle] (a) at (90:5)      {$\rep{\mathbf{3}}\oplus\rep{2\\3}\oplus\rep{1\\2}$};
  \node[draw,shape=rectangle] (b) at (90+72:5)   {$\rep{\mathbf{3}}\oplus\rep{1\\2}\oplus\rep{1}$};
  \node[draw,shape=rectangle] (c) at (90+72*2:5) {$\rep{\mathbf{3}}\oplus\rep{1}$};
  \node[draw,shape=rectangle] (d) at (90+72*3:5) {$\rep{\mathbf{3}}$};
  \node[draw,shape=rectangle] (e) at (90+72*4:5) {$\rep{\mathbf{3}}\oplus\rep{2\\3}$};
  \node (f) at (90:3)      {$\rep{2}\oplus\rep{2\\3}\oplus\rep{1\\2}$};
  \node (g) at (90+72:3)   {$\rep{1\\2}\oplus\rep{1}$};
  \node (h) at (90+72*2:3) {$\rep{1}$};
  \node (i) at (90+72*3:3) {$\rep{0}$};
  \node (j) at (90+72*4:3) {$\rep{2}\oplus\rep{2\\3}$};
  \node (k) at (90+36:3*0.809017)    {$\rep{2}\oplus\rep{1\\2}$};
  \node (l) at (90-72-36:3*0.809017) {$\rep{2}$};
  \path[commutative diagrams/.cd, every arrow]
  (a) edge (b)
  (b) edge (c)
  (c) edge (d)
  (a) edge (e)
  (e) edge (d)
  (a) edge (f)
  (f) edge (k)
  (k) edge (g)
  (g) edge (h)
  (h) edge (i)
  (f) edge (j)
  (j) edge (l)
  (l) edge (i)
  (k) edge (l)
  (b) edge (g)
  (e) edge (j)
  (d) edge (i)
  (c) edge (h);
\end{tikzpicture}
%%% Local Variables:
%%% TeX-master: "master"
%%% End:
}
    \caption{Embedding of $\sttilt C$ in $Q(\sttilt A)$, 
      see Example \ref{ex:A3r2-catU}.} 
    \label{fig:exchange-graph-sttilt-A3r2}
  \end{figure}
\end{example}

\begin{example}
  \label{ex:preproj}
  Let $A$ be the preprojective algebra of Dynkin type $A_3$, \ie
  the algebra given by the quiver 
  \begin{center}
    \begin{tikzcd}
      3\rar[bend right=25,swap]{y_2} & 
      2 \lar[bend right=25,swap]{x_2}\rar[bend right=25,swap]{y_1}
      & 1 \lar[bend right=25,swap]{x_1}
    \end{tikzcd}    
  \end{center}  
  with relations $x_1y_1=0$, $y_2x_2=0$ and $y_1x_1=x_2y_2$.

  Let $U=\rep{2\\1}$, then $U$ is $\tau$-rigid and not a support
  $\tau$-tilting $A$-module. 
  The Bongartz completion of $U$ is given by
  $T=P_1\oplus P_2\oplus\rep{2\\1}
  =\rep{1\\2\\3}\oplus\rep{2\\13\\2}\oplus\rep{2\\1}$;
  hence $C$ is isomorphic to the path algebra given by the quiver
  
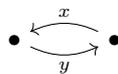
\begin{center}
    \begin{tikzcd}
      \bullet\rar[bend right=25,swap]{y} & 
      \bullet \lar[bend right=25,swap]{x}
    \end{tikzcd}    
  \end{center}
  with the relations $xy=0$ and $yx=0$, that is $C$ is isomorphic to
  the preprojective algebra of type $A_2$.
  In this case $\lperp{(\tau U)}$ consists of all
  $A$-modules $M$ such that $\tau U = S_3$ is not a direct
  summand of $\mathrm{top}\, M$.
  On the other hand, it is easy to see that the only
  indecomposable $A$-modules in $\lperp{(\tau U)}$ which do not belong
  to $U^\perp$ are $U$, $P_2$, $S_2$ and $\rep{1\\2}$.
  We can visualize this in the Auslander-Reiten quiver of $\mod A$
  as follows (note that the dashed edges are to be identified to
  form a M\"obius strip):

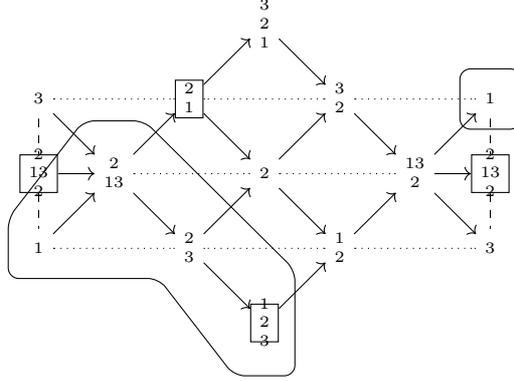
  \begin{center}
    \begin{tikzpicture}[commutative diagrams/every diagram]
    \node (P1) at (mesh cs:u=0,v=0,r=0.7) {$\rep{2\\13}$};
    \node (S2) at (mesh cs:u=-1,v=0,r=0.7) {$\rep{3}$};
    \node (X) at (mesh cs:u=0,v=2,r=0.7) {$\rep{3\\2\\1}$};
    \node[draw,rectangle] (Y) at (mesh cs:u=2,v=0,r=0.7) {$\rep{1\\2\\3}$};
    \node[draw,rectangle] (S3) at (mesh cs:u=-1,v=-1,r=0.35) {$\rep{2\\13\\2}$};
    \node (S4) at (mesh cs:u=0,v=-1,r=0.7) {$\rep{1}$};

    \node (M1) at (mesh cs:u=1,v=1,r=0.7) {$\rep{2}$};
    \node[draw,rectangle] (M2) at (mesh cs:u=0,v=1,r=0.7) {$\rep{2\\1}$};
    % \node (M3) at (mesh cs:u=1,v=1,r=0.35) {$\bullet$};
    \node (M4) at (mesh cs:u=1,v=0,r=0.7) {$\rep{2\\3}$};

    \node (S1) at (mesh cs:u=2,v=2,r=0.7) {$\rep{13\\2}$};
    \node (I2) at (mesh cs:u=1,v=2,r=0.7) {$\rep{3\\2}$};
    % \node (I3) at ($(mesh cs:u=1,v=1,r=0.7)+(mesh
    % cs:u=1,v=1,r=0.35)$) {$T_3$}; 
    \node (I4) at (mesh cs:u=2,v=1,r=0.7) {$\rep{1\\2}$};
    
    % \node (TP1) at (mesh cs:u=3,v=3,r=0.7) {$\bullet$};
    \node (TS2) at (mesh cs:u=2,v=3,r=0.7) {$\rep{1}$};
    \node[draw,rectangle] (TS3) at ($(mesh cs:u=2,v=2,r=0.7)+(mesh cs:u=1,v=1,r=0.35)$) {$\rep{2\\13\\2}$}; 
    \node (TS4) at (mesh cs:u=3,v=2,r=0.7) {$\rep{3}$};

    % \node (S2a) at (mesh cs:u=3,v=4,r=0.7) {$\bullet$};
    % \node (S3a) at ($(mesh cs:u=3,v=3,r=0.7)+(mesh cs:u=1,v=1,r=0.35)$) {$T_1$}; 
    % \node (S4a) at (mesh cs:u=4,v=3,r=0.7) {$\bullet$};

    \draw[rounded corners=4pt] ($(P1)+(-0.4,0.7)$) --
    ($(P1)+(0.4,0.7)$) --  ($(Y)+(0.4,0.7)$) --
    ($(Y)+(0.4,-0.7)$) -- ($(Y)-(0.4,0.7)$) --
    ($(M4)-(0.4,0.4)$) -- ($(S4)-(0.4,0.4)$) --
    ($(S4)+(-0.4,0.4)$) -- cycle;

    \draw[rounded corners=4pt] ($(TS2)+(0.4,0.4)$) --
    ($(TS2)+(0.4,-0.4)$) --
    ($(TS2)-(0.4,0.4)$) --
    ($(TS2)+(-0.4,0.4)$) -- cycle;

    \path[commutative diagrams/.cd, every arrow]
    (S2) edge (P1)

    (S4) edge (P1)
    (P1) edge (M2)

    (P1) edge (M4)
    (M2) edge (M1)

    (M4) edge (M1)
    (M1) edge (I2)

    (M1) edge (I4)
    (I2) edge (S1)

    (I4) edge (S1)
    (S1) edge (TS2)

    (S1) edge (TS4)
    
    (M2) edge (X)
    (X) edge (I2)

    (M4) edge (Y)
    (Y) edge (I4)

    (S3) edge (P1)
    (S1) edge (TS3)

    (P1)  edge[commutative diagrams/path, dotted] (M1)
    (M1)  edge[commutative diagrams/path, dotted] (S1)

    (S2)  edge[commutative diagrams/path, dotted] (M2)
    (I2)  edge[commutative diagrams/path, dotted] (M2)
    (I2)  edge[commutative diagrams/path, dotted] (TS2)

    (S4)  edge[commutative diagrams/path, dotted] (M4)
    (I4)  edge[commutative diagrams/path, dotted] (M4)
    (I4)  edge[commutative diagrams/path, dotted] (TS4)

    (S2)  edge[commutative diagrams/path, dashed] (S3)
    (S3)  edge[commutative diagrams/path, dashed] (S4)
    (TS2)  edge[commutative diagrams/path, dashed] (TS3)
    (TS3)  edge[commutative diagrams/path, dashed] (TS4);
  \end{tikzpicture}
  \end{center}

  The indecomposable summands of $T$ are indicated with rectangles
  and $\U$ is encircled.
  Note that $\U$ is equivalent
  to $\mod C$ as shown in Theorem \ref{thm:U-modC}.
  By Theorem \ref{thm:tau-tilting-reduction} we have that 
  $\sttilt C$ can be embedded as an interval in $\sttilt A$.
  We have indicated this embedding in $Q(\sttilt A)$ in Figure
  \ref{fig:exchange-graph-preproj} by enclosing the image of $\sttilt
  C$ in rectangles.
  \begin{figure}[t]
    \centering
    \scalebox{0.75}{\begin{tikzpicture}
      \node (5) at (90:2)      {$\rep{3}\oplus\rep{11\\3}\oplus\rep{3\\2\\1}$};
      \node (6) at (90+60:2)   {$\rep{3}\oplus\rep{3\\2}\oplus\rep{3\\2\\1}$};
      \node (16) at (90+60*2:2) {$\rep{3}\oplus\rep{3\\2}$};
      \node (23) at (90+60*3:2) {$\rep{3}$};
      \node (20) at (90+60*4:2) {$\rep{3}\oplus\rep{1}$};
      \node (15) at (90+60*5:2) {$\rep{3}\oplus\rep{13\\2}\oplus\rep{1}$};
      \node (1) at (90:5.5)      {$\rep{1\\2\\3}\oplus\rep{2\\13\\2}\oplus\rep{3\\2\\1}$};
      \node (2) at (90+60:5.5)   {$\rep{2\\3}\oplus\rep{2\\13\\2}\oplus\rep{3\\2\\1}$};
      \node (17) at (90+60*2:5.5) {$\rep{2\\3}\oplus\rep{2}$};
      \node (21) at (90+60*3:5.5) {$\rep{2}$};
      \node (19) at (90+60*4:5.5) {$\rep{1\\2}\oplus\rep{1}$};
      \node (11) at (90+60*5:5.5) {$\rep{1\\2\\3}\oplus\rep{1\\2}\oplus\rep{1}$};
      \node[draw,shape=rectangle] (3) at (90:8)      {$\rep{1\\2\\3}\oplus\rep{2\\13\\2}\oplus\rep{2\\1}$};
      \node[draw,shape=rectangle] (8) at (90+60:8)   {$\rep{2\\3}\oplus\rep{2\\13\\2}\oplus\rep{2\\1}$};
      \node[draw,shape=rectangle] (13) at (90+60*2:8) {$\rep{2\\3}\oplus\rep{2}\oplus\rep{2\\1}$};
      \node[draw,shape=rectangle] (18) at (90+60*3:8) {$\rep{2}\oplus\rep{2\\1}$};
      \node[draw,shape=rectangle] (14) at (90+60*4:8) {$\rep{1\\2}\oplus\rep{2\\1}$};
      \node[draw,shape=rectangle] (9) at (90+60*5:8) {$\rep{1\\2\\3}\oplus\rep{1\\2}\oplus\rep{2\\1}$};
      \node (7) at (-0.86*4.5,2*0.5) {$\rep{2\\3}\oplus\rep{3\\2}\oplus\rep{3\\2\\1}$};
      \node (12) at (-0.86*4.5,-2*0.5) {$\rep{2\\3}\oplus\rep{3\\2}$};
      \node (4) at ($(5)+(2.5*0.86*0.5,2.5*0.86*0.86)$) {$\rep{1\\2\\3}\oplus\rep{13\\2}\oplus\rep{3\\2\\1}$};
      \node (10) at ($(15)+(2.5*0.86*0.5,2.5*0.86*0.86)$) {$\rep{1\\2\\3}\oplus\rep{13\\2}\oplus\rep{1}$};
      \node (22) at ($(20)+(2.5*0.86*0.5,-2.5*0.86*0.86)$) {$\rep{1}$};
      \node (24) at ($(23)+(2.5*0.86*0.5,-2.5*0.86*0.86)$) {$\rep{0}$};
      \path[commutative diagrams/.cd, every arrow]
      (3) edge (8)
      (3) edge (9)
      (9) edge (14)
      (8) edge (13)
      (13) edge (18)
      (14) edge (18)
      (1) edge (2)
      (17) edge (21)
      (11) edge (19)
      (5) edge (6)
      (5) edge (15)
      (6) edge (16)
      (15) edge (20)
      (16) edge (23)
      (20) edge (23)
      (1) edge (3)
      (2) edge (8)
      (13) edge (17)
      (18) edge (21)
      (14) edge (19)
      (9) edge (11)
      (2) edge (7)
      (7) edge (12)
      (12) edge (17)
      (6) edge (7)
      (12) edge (16)
      (1) edge (4)
      (4) edge (10)
      (10) edge (11)
      (4) edge (5)
      (10) edge (15)
      (20) edge (22)
      (23) edge (24)
      (21) edge (24)
      (19) edge (22)
      (22) edge (24);
\end{tikzpicture}

%%% Local Variables: 
%%% mode: latex
%%% TeX-master: "../master"
%%% End: 
}
    \caption{Embedding of $\sttilt C$ in $Q(\sttilt A)$, 
      see Example \ref{ex:preproj}.} 
    \label{fig:exchange-graph-preproj}
  \end{figure}  
\end{example}

\begin{example}
  \label{ex:kronecker-ex}
  Let $A$ be the algebra given by the path algebra of the quiver
  
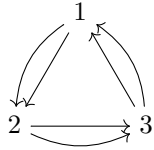
\begin{center}
    \begin{tikzpicture}[commutative diagrams/every diagram]
      \node (A) at (90+120:1){$2$};
      \node (B) at (90:1){$1$};
      \node (C) at (90+120*2:1){$3$};
      \path[commutative diagrams/every arrow]
      (B) edge[bend right=25] (A)
      (B) edge (A)
      (C) edge[bend right=25] (B)
      (C) edge (B)
      (A) edge[bend right=25] (C)
      (A) edge (C);
    \end{tikzpicture}        
  \end{center}
  modulo the ideal generated by all paths of length two.

  Let $U=\rep{11\\222}$. The Bongartz completion
  of $U$ is given by
  $T=P_1\oplus \rep{11\\222}\oplus P_3
  =\rep{1\\22}\oplus\rep{11\\222}\oplus\rep{3\\11}$;
  hence $C\cong k\times k$.
  It is easy to see that $Q(\sttilt C)$ is given by the quiver
  
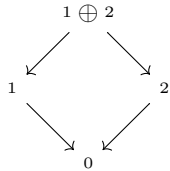
\begin{center}
    \begin{tikzpicture}
      \node (kk) at (90:1) {$\rep{1}\oplus\rep{2}$};
      \node (k0) at (90+90:1) {$\rep{1}$};
      \node (0) at (270:1) {$\rep{0}$};
      \node (0k) at (0:1) {$\rep{2}$};
      \path[commutative diagrams/every arrow]
      (kk) edge (k0)
      (kk) edge (0k)
      (k0) edge (0)
      (0k) edge (0);
    \end{tikzpicture}
  \end{center}
  By Theorem \ref{thm:tau-tilting-reduction} we have that 
  $\sttilt C$ can be embedded as an interval in $\sttilt A$.
  We have indicated this embedding in $Q(\sttilt A)$ in Figure
  \ref{fig:exchange-graph-kronecker} by drawing $Q(\sttilt C)$
  with double arrows.
  \begin{figure}[t]
    \centering
    \scalebox{0.75}{\begin{tikzpicture}
  \node (1) at (90:8) {$\rep{1\\22}\oplus\rep{2\\33}\oplus\rep{3\\11}$};
  \node (2) at (90+16.36:8) {$\rep{33\\111}\oplus\rep{2\\33}\oplus\rep{3\\11}$};
  \node (3) at (90+16.36*2:8) {$\rep{33\\111}\oplus\rep{2\\33}\oplus\rep{333\\1111}$};
  \node (4) at (90+16.36*3:8) {\rotatebox[origin=c]{50}{$\cdots$}};
  \node (5) at (90+16.36*4:8) {$\rep{3}\oplus\rep{2\\33}\oplus\rep{33\\1}$};
  \node (6) at (90+16.36*5:8) {$\rep{3}\oplus\rep{2\\33}$};
  \node (7) at (90+16.36*6:8) {$\rep{22\\333}\oplus\rep{2\\33}$};
  \node (8) at (90+16.36*7:8) {$\rep{22\\333}\oplus\rep{222\\3333}$};
  \node (9) at (90+16.36*8:8) {\rotatebox[origin=c]{-50}{$\cdots$}};
  \node (10) at (90+16.36*9:8) {$\rep{22\\3}\oplus\rep{2}$};
  \node (11) at (90+16.36*10:8) {$\rep{2}$};
  \node (12) at (-90:8) {$\rep{0}$};
  \node (2a) at (90-16.36:8) {$\rep{1\\22}\oplus\rep{2\\33}\oplus\rep{22\\333}$};
  \node (3a) at (90-16.36*2:8) {$\rep{1\\22}\oplus\rep{222\\3333}\oplus\rep{22\\333}$};
  \node (4a) at (90-16.36*3:8) {\rotatebox[origin=c]{-50}{$\cdots$}};
  \node (5a) at (90-16.36*4:8) {$\rep{1\\22}\oplus\rep{22\\3}\oplus\rep{2}$};
  \node (6a) at (90-16.36*5:8) {$\rep{1\\22}\oplus\rep{2}$};
  \node (7a) at (90-16.36*6:8) {$\rep{1\\22}\oplus\rep{11\\222}$};
  \node (8a) at (90-16.36*7:8) {$\rep{111\\2222}\oplus\rep{11\\222}$};
  \node (9a) at (90-16.36*8:8) {\rotatebox[origin=c]{50}{$\cdots$}};
  \node (10a) at (90-16.36*9:8) {$\rep{11\\2}\oplus\rep{1}$};
  \node (11a) at (90-16.36*10:8) {$\rep{1}$};
  \path[commutative diagrams/.cd, every arrow]
  (1) edge (2)
  (2) edge (3)
  (3) edge (4)
  (4) edge (5)
  (5) edge (6)
  (6) edge (7)
  (7) edge (8)
  (8) edge (9)
  (9) edge (10)
  (10) edge (11)
  (11) edge (12)
  (1) edge (2a)
  (2a) edge (3a)
  (3a) edge (4a)
  (4a) edge (5a)
  (5a) edge (6a)
  (6a) edge (7a)
  (7a) edge[commutative diagrams/Rightarrow] (8a)
  (8a) edge (9a)
  (9a) edge (10a)
  (10a) edge (11a)
  (11a) edge (12);
  \path[commutative diagrams/.cd, every arrow]
  (2a) edge (7)
  (3a) edge (8)
  (5a) edge (10)
  (6a) edge (11);
  \node[white] (2b) at (0,1.5*5-1) {$\rep{1\\22}\oplus\rep{11\\222}\oplus\rep{3\\11}$};
  \node[white] (3b) at (0,1.5*4-1) {$\rep{111\\2222}\oplus\rep{11\\222}\oplus\rep{3\\11}$};
  \node[white] (4b) at (0,1.5*3-1) {$\vdots$};
  \node[white] (5b) at (0,1.5*2-1) {$\rep{11\\2}\oplus\rep{111\\22}\oplus\rep{3\\11}$};
  \node[white] (6b) at (0,1.5*1-1) {$\rep{1}\oplus\rep{3\\11}$};
  \node[white] (7b) at (0,0-1) {$\rep{33\\111}\oplus\rep{3\\11}$};
  \node[white] (8b) at (0,-1.5*1-1) {$\rep{33\\111}\oplus\rep{333\\1111}$};
  \node[white] (9b) at (0,-1.5*2-1) {$\vdots$};
  \node[white] (10b) at (0,-1.5*3-1) {$\rep{33\\1}\oplus\rep{3}$};
  \node[white] (11b) at (0,-1.5*4-1) {$\rep{3}$};
  \path[commutative diagrams/.cd, every arrow]
  (1) edge[commutative diagrams/crossing over] (2b)
  (2b) edge[commutative diagrams/crossing over,commutative diagrams/Rightarrow] (3b)
  (3b) edge[commutative diagrams/crossing over] (4b)
  (4b) edge[commutative diagrams/crossing over] (5b)
  (5b) edge[commutative diagrams/crossing over] (6b)
  (6b) edge[commutative diagrams/crossing over] (7b)
  (7b) edge[commutative diagrams/crossing over] (8b)
  (8b) edge[commutative diagrams/crossing over] (9b)
  (9b) edge[commutative diagrams/crossing over] (10b)
  (10b) edge[commutative diagrams/crossing over] (11b)
  (11b) edge[commutative diagrams/crossing over] (12);
  \path[commutative diagrams/.cd, every arrow]
  (2b) edge[commutative diagrams/crossing over,commutative diagrams/Rightarrow] (7a)
  (3b) edge[commutative diagrams/crossing over,commutative diagrams/Rightarrow] (8a)
  (5b) edge[commutative diagrams/crossing over] (10a)
  (6b) edge[commutative diagrams/crossing over] (11a);
  \path[commutative diagrams/.cd, every arrow]
  (2) edge[commutative diagrams/crossing over] (7b)
  (3) edge[commutative diagrams/crossing over] (8b)
  (5) edge[commutative diagrams/crossing over] (10b)
  (6) edge[commutative diagrams/crossing over] (11b);
  ;
  \node[fill=white] (2d) at (0,1.5*5-1) {$\rep{1\\22}\oplus\rep{11\\222}\oplus\rep{3\\11}$};
  \node[fill=white] (3d) at (0,1.5*4-1) {$\rep{111\\2222}\oplus\rep{11\\222}\oplus\rep{3\\11}$};
  \node[fill=white] (4d) at (0,1.5*3-1) {$\vdots$};
  \node[fill=white] (5d) at (0,1.5*2-1) {$\rep{11\\2}\oplus\rep{111\\22}\oplus\rep{3\\11}$};
  \node[fill=white] (6d) at (0,1.5*1-1) {$\rep{1}\oplus\rep{3\\11}$};
  \node[fill=white] (7d) at (0,0-1) {$\rep{33\\111}\oplus\rep{3\\11}$};
  \node[fill=white] (8d) at (0,-1.5*1-1) {$\rep{33\\111}\oplus\rep{333\\1111}$};
  \node[fill=white] (9d) at (0,-1.5*2-1) {$\vdots$};
  \node[fill=white] (10d) at (0,-1.5*3-1) {$\rep{33\\1}\oplus\rep{3}$};
  \node[fill=white] (11d) at (0,-1.5*4-1) {$\rep{3}$};
\end{tikzpicture}

%%% Local Variables: 
%%% mode: latex
%%% TeX-master: "../master"
%%% End: 
}
    \caption{Embedding of $\sttilt C$ in $Q(\sttilt A)$, 
      see Example \ref{ex:kronecker-ex}.} 
    \label{fig:exchange-graph-kronecker}
  \end{figure}
\end{example}

\subsection{Proof of the main theorems}
\label{sec:proofs-of-main-thms}

We begin with the proof of Theorem \ref{thm:tp-reduction}.
The following proposition shows that the map 
$\T\mapsto \T\cap U^\perp$ in Theorem \ref{thm:tp-reduction} is
well defined. 

\begin{proposition}
  \label{prop:ff-torsion-classes}
  Let $\T$ be a torsion class in $\mod A$ 
  such that $\Fac U\subseteq \T\subseteq \lperp{(\tau U)}$. Then
  the following holds:
  \begin{enumerate}
  \item $\T\cap U^\perp$ is in $\tors\U$.
  \item[(b)] $\T\cap U^\perp=\f\T$.
  \end{enumerate}
  If in addition $\T$ is functorially finite in $\mod A$, then we
  have:
  \begin{enumerate}
  \item[(c)] $\T\cap U^\perp = \Fac(\f P(\T))\cap U^\perp$.
  \item[(d)] $F(\T\cap U^\perp)$ is in $\ftors C$.
  \end{enumerate}
\end{proposition}
\begin{proof}
  (a) $\T\cap U^\perp$ is closed under extensions since both 
  $\T$ and $U^\perp$ are closed under extensions in $\mod A$. Now
  let $0\to L\to M\to N\to 0$ be a short exact sequence in $\mod A$
  with terms in $\U$. If
  $M\in\T$, then $N\in\T$ since $\T$ is closed under factor
  modules. Thus $N\in\T\cap\U = \T\cap U^\perp$. 
  This shows that $\T\cap U^\perp$ is a torsion class in $\U$.
  
  (b) Since $\T$ is closed under factor modules in $\mod A$ we
  have that $\f\T\subseteq \T\cap U^\perp$, hence we only need to
  show the reverse inclusion. 
  Let $M\in\T\cap U^\perp$. 
  In particular we have that $M\in U^\perp$, hence $\f M= M$ and
  the claim follows.

  (c) Since $\Fac(\f P(\T))\subseteq \T$, we have that 
  $\Fac(\f P(\T))\cap U^\perp\subseteq\T\cap U^\perp$.
  Now we show the opposite inclusion.
  Let $M$ be in $\T\cap U^\perp$, then there is an epimorphism
  $f:X\to M$ with $X$ in $\add(P(\T))$. 
  Since there are no non-zero morphisms from
  $\Fac U$ to $U^\perp$ we have a commutative diagram 
  \begin{center}
    \begin{tikzcd}
      0 \rar & \t X \rar \drar[swap]{0} & X \rar \dar[two heads]{f} &
      \f X \rar \dlar[dashrightarrow] & 0 \\ 
      & & M
    \end{tikzcd}    
  \end{center}
  Hence $M$ is in $\Fac(\f P(\T))\cap U^\perp$ and we have the
  equality $\T\cap U^\perp = \Fac(\f P(\T))\cap U^\perp$.
    
  (d) 
  By Proposition \ref{prop:technical-reasons} we have that 
  $\f P(\T)$ is the $\Ext$-progenerator of
  $\T\cap U^\perp\subseteq\U$, and since $F:\U\to\mod C$ is an
  exact equivalence, see Theorem \ref{thm:U-modC}, we have that
  $F(\f P(\T))$ is the $\Ext$-progenerator of $F(\T\cap U^\perp)$.
  Then by Proposition \ref{prop:torsion-pairs} we have that
  $F(\T\cap U^\perp)=\Fac(F(\f P(\T)))$ is functorially finite in
  $\mod C$.
\end{proof}

Now we consider the converse map $\G\mapsto (\Fac U)*\G$.
We start with the following easy observation.

\begin{lemma}
  \label{lem:canonical-sequence}
  Let $\G$ be in $\tors\U$ and $M$ be an $A$-module. 
  Then $M$ is in $(\Fac U)*\G$ if and only if $\f M$ belongs to
  $\G$. 
  In particular, if $M$ is in $(\Fac U)*\G$ then 
  $0\to \t M\to M\to \f M\to 0$ is the unique way to express $M$
  as an extension of a module from $\Fac U$ by a module from
  $\G$.
\end{lemma}
\begin{proof}
  Both claims follow immediately from the uniqueness of the 
  canonical sequence.
\end{proof}

The following lemma gives canonical sequences in $\U$.

\begin{lemma}
  \label{lem:can-seq-U}
  Let $\G$ be a torsion class in $\U$. 
  Then there exist functors $\t_\G:\U\to \G$ and
  $\f_\G:\U\to \G^\perp\cap \U$ and natural transformations
  $\t_\G\to 1_\U\to \f_\G$ such that the sequence 
  \[
  0\to \t_\G M\to M\to \f_\G M\to 0
  \]
  is exact in $\mod A$ for each $M$ in $\U$.
\end{lemma}
\begin{proof}
  Since $(F\G,F(\G^\perp\cap \U))$ is a torsion pair in
  $\mod C$ by Theorem \ref{thm:U-modC}, we have associated
  canonical sequences in $\mod C$. 
  Applying the functor $G$, we get the desired functors.
\end{proof}

The following proposition shows that the map 
$\G\mapsto (\Fac U)*\G$ in Theorem \ref{thm:ff-tp-reduction} is
well-defined. 

\begin{proposition}
  \label{prop:ast-ff-torsion-class}
  Let $\G$ be in $\tors\U$. Then $(\Fac U)*\G$ is a torsion 
  class in $\mod A$ such that
  $\Fac U \subseteq (\Fac U)*\G \subseteq \lperp{(\tau U)}$.
\end{proposition}
\begin{proof}
  First, it is clear that $\Fac U \subseteq (\Fac U)*\G
  \subseteq\lperp{(\tau U)}$ since $\G$ and $\Fac U$ are
  subcategories of $\lperp{(\tau U)}$ and $\lperp{(\tau U)}$ is
  closed under extensions.

  Now, let us show that $(\Fac U)*\G$ is closed under factor
  modules. 
  Let $M$ be in $(\Fac U)*\G$, so by 
  Lemma \ref{lem:canonical-sequence} we have that $\f M$ is in
  $\G$, and let $f:M\to N$ be an epimorphism. 
  We have a commutative diagram
  \begin{center}
    \begin{tikzcd}
      0 \rar & \t M \rar \dar{\t f} &
      M \rar \dar{f} & \f M \rar \dar{\f f} & 0 \\ 
      0 \rar & \t N \rar & N \rar \dar & \f N \rar
      \dar & 0 \\ 
      & & 0 & 0
    \end{tikzcd}    
  \end{center}
  with exact rows and columns. 
  Since $\f M\in\G\subseteq\lperp{(\tau U)}$, we have that 
  $\f N\in\lperp{(\tau U)}$. 
  Thus $\f N\in\lperp{(\tau U)}\cap U^\perp=\U$. 
  Since $\f M$ belong to $\G$ which is a torsion class in $\U$, we
  have that $\f N\in\G$. 
  Finally, since $\t N\in\Fac U$ we have that $N\in(\Fac U)*\G$. 
  The claim follows.
  
  To show that $(\Fac U)*\G$ is closed under extensions it is 
  sufficient to show that $\G*\Fac U\subseteq (\Fac U)*\G$ since
  this implies 
  \begin{align*}
   ((\Fac U)*\G)*((\Fac U)*\G)=&(\Fac U)*(\G*(\Fac U))*\G\\
  \subseteq& (\Fac U)*(\Fac U)*\G*\G
  =(\Fac U)*\G
  \end{align*}
  by the associativity of the operation $*$. 
  For this, let $0\to N\to M\to L\to 0$ be a short exact sequence 
  with $N$ in $\G$ and $L$ in $\Fac U$. 
  We only have to show that $\f M\in\G$, or equivalently that
  $\f_\G(\f M)=0$.
  Since $N\in \G$ and $\f_\G(\f M)\in\G^\perp$ we have
  $\Hom_A(N,\f_\G(\f M))=0$.
  Also, $\Hom_A(L,\f_\G(\f M))=0$ since $L\in\Fac U$ and
  $\f_\G(\f M)\in U^\perp$.
  Thus we have $\Hom_A(M,\f_\G(\f M))=0$.
  But $\f_\G(\f M)$ is a factor module of $M$ so we have that
  $\f_\G(\f M)=0$.
  Thus $\f M = \t_\G(\f M)$ belongs to $\G$. 
\end{proof}

Now we give the proof Theorem \ref{thm:tp-reduction}.

\begin{proof}[Proof of Theorem \ref{thm:tp-reduction}]
  By Corollary \ref{cor:torsU} we have that the functors
  $F$ and $G$ induce mutually inverse bijections between $\tors\U$
  and $\tors C$. 
  It follows from  Proposition \ref{prop:ff-torsion-classes}(a)
  that the correspondence $\T\mapsto \T\cap U^\perp$ gives a well
  defined map
  \[
  \setP{\T\in\tors A}{\Fac U\subseteq\T\subseteq \lperp{(\tau
      U)}} \longrightarrow \tors\U.
  \]
  On the other hand, it follows from 
  Proposition \ref{prop:ast-ff-torsion-class} that the association
  $\G \mapsto
  (\Fac U)*\G$ gives a well defined map
  \[
  \tors\U \longrightarrow \setP{\T\in\tors A}{\Fac U\subseteq
  \T\subseteq \lperp{(\tau U)}}.
  \]

  It remains to show that the maps
  \[
  \T\mapsto\T\cap U^\perp \quad\text{and}\quad \G\mapsto \Fac U
  *\G 
  \]
  are inverse of each other. 
  Let $\T$ be a torsion class in $\mod A$ such that 
  $\Fac U\subseteq\T\subseteq \lperp{(\tau U)}$.
  Since $\T$ is closed under extensions, we have that
  $(\Fac U)*(\T\cap U^\perp)\subseteq \T$. 
  Thus we only need to show the opposite inclusion.
  Let $M$ be in $\T$, then we have an exact sequence 
  \[
  0 \to \t M \to M \to \f M \to 0
  \]
  with $\t M\in\Fac U$ and $\f M$ in $\T\cap U^\perp$ since $\T$
  is closed under factor modules.
  Thus $M\in (\Fac U)\ast(\T\cap U^\perp)$ holds and the claim
  follows. 
  
  On the other hand, let $\G$ be a torsion class in $\U$. 
  It is clear that $\G \subseteq ((\Fac U)*\G)\cap U^\perp$, so we
  only need to show the opposite inclusion. 
  But if $M$ is in $((\Fac U)*\G)\cap U^\perp$, then 
  $M\in U^\perp$ implies that $M\cong \f M$.
  Moreover, by Lemma \ref{lem:canonical-sequence} we have that 
  $M\cong \f M$ belongs to $\G$. 
  This finishes the proof of the theorem.
\end{proof}

Now we begin to prove Theorem \ref{thm:ff-tp-reduction}. 
For this we need the following technical result:

\begin{proposition}
  \cite[Prop 5.33]{iyama_stable_2013}
  \label{prop:ast-ff}
  Let $\X$ and $\Y$ be covariantly finite subcategories of
  $\mod A$. Then $\X*\Y$ is also covariantly finite in $\mod A$.  
\end{proposition}

We also need the following observation.

\begin{lemma}
  \label{lem:U-covariantly-finite}
  $\U$ is covariantly finite in $\lperp{(\tau U)}$.
\end{lemma}
\begin{proof}
  Let $M$ be in $\lperp{(\tau U)}$ and consider the canonical
  sequence 
  \[
  0\to \t M\to M\xto{f} \f M \to 0.
  \]
  Then $\f M$ is in $\U$ by \eqref{eq:f-lperp} and clearly $f$ is
  a left  $\U$-approximation (mind that $\U\subseteq U^\perp$). 
  Thus $\U$ is covariantly finite in $\lperp{(\tau U)}$ as
  required.
\end{proof}

The following proposition shows that the map
$\G\mapsto (\Fac U)*\G$ in Theorem \ref{thm:ff-tp-reduction}
preserves functorial finiteness, and thus is well defined.

\begin{proposition}
  \label{prop:ast-ff-2}
  Let $\G$ be in $\ftors\U$. 
  Then $(\Fac U)*\G$ is a functorially finite torsion class in
  $\mod A$ such that 
  $\Fac U\subseteq (\Fac U)*\G \subseteq \lperp{(\tau U)}$. 
\end{proposition}
\begin{proof}
  By Proposition \ref{prop:ast-ff-torsion-class}, we only need to
  show that $(\Fac U)*\G$ is covariantly finite in $\mod A$. 
  Since $\Fac U$ is covariantly finite in $\mod A$, see 
  Proposition \ref{prop:torsion-pairs}(b), 
  by Proposition \ref{prop:ast-ff} it is enough to show that $\G$
  is covariantly finite in $\mod A$. 
  By Lemma \ref{lem:U-covariantly-finite} we have that $\U$ is
  covariantly finite in $\lperp{(\tau U)}$. 
  Since $\G$ is covariantly finite in $\U$ and $\lperp{(\tau U)}$ 
  is covariantly finite in $\mod A$, 
  see Proposition \ref{prop:torsion-pairs}(b),  we have that  $\G$
  is covariantly finite in $\mod A$. 
\end{proof}

We are ready to give the proof of 
Theorem \ref{thm:ff-tp-reduction}. 

\begin{proof}[Proof of Theorem \ref{thm:ff-tp-reduction}] 
  We only need to show that the bijections in Theorem
  \ref{thm:tp-reduction} preserve functorial finiteness. 
  But this follows immediately from Proposition
  \ref{prop:ff-torsion-classes}(d) and 
  Proposition \ref{prop:ast-ff-2}.   
  The theorem follows.
\end{proof}

%%% Local Variables: 
%%% mode: latex
%%% TeX-master: "../master"
%%% End: 

\section{Compatibility with other types of reduction}
\label{sec:compatibility}

Let $A$ be a finite dimensional algebra. 
Then support $\tau$-tilting $A$-modules are in bijective
correspondence with the so-called two-term silting complexes in 
$K^{\mathrm{b}}(\proj A)$, 
see \cite[Thm. 3.2]{adachi_tau-tilting_2012}.

On the other hand, if $A$ is a 2-Calabi-Yau-tilted algebra from a
2-Calabi-Yau category $\C$, then there is a bijection between 
$\sttilt A$ and the set of isomorphism classes of basic
cluster-tilting objects in $\C$, 
see \cite[Thm. 4.1]{adachi_tau-tilting_2012}.

Reduction techniques were established (in greater generality) in
\cite[Thm. 4.9]{iyama_mutation_2008} for cluster-tilting objects
and for silting objects in \cite[Thm. 2.37]{aihara_silting_2012}
for a special case and in \cite{iyama_silting_2013} for the
general case.
The aim of this section is to show that these reductions are
compatible with $\tau$-tilting reduction as established in
Section \ref{sec:reduction}. 

Given two subcategories $\X$ and $\Y$ of a triangulated category
$\T$, we write $\X\ast\Y$ for the full subcategory of $\T$
consisting of all objects $Z\in\T$ such that there exists a
triangle
\[
X\to Z\to Y\to X[1]
\]
with $X\in\X$ and $Y\in\Y$. 
For objects $X$ and $Y$ in $\T$ we define 
$X\ast Y:=(\add X)\ast(\add Y)$.

\subsection{Silting reduction}
\label{sec:silting-compatibility}

Let $\T$ be a Krull-Schmidt triangulated category and $S$ an
object in $\T$.
Following \cite[Def. 2.1]{aihara_silting_2012}, we say that $M$ is 
a \emph{presilting object in $\T$} if
\[
\Hom_\T(M,M[i])=0\quad\text{for all }i>0.
\]
We call $S$ a \emph{silting object} if moreover $\thick(S) = \T$,
where $\thick(S)$ is the smallest triangulated subcategory of $\T$
which contains $S$ and is closed under direct summands and
isomorphisms.
We denote the set of isomorphism classes of all basic silting
objects in $\T$ by $\silt \T$. 

Let $M,N\in\silt\T$. 
We write $N\leqslant M$ if and only if $\Hom_\T(M,N[i])=0$ for
each $i>0$.
Then $\leqslant$ is a partial order in $\silt \T$, see
\cite[Thm. 2.11]{aihara_silting_2012}.  

\begin{setting}
  \label{set:silting}
  We fix a $k$-linear, $\Hom$-finite, Krull-Schmidt triangulated
  category $\T$ with a silting object $S$, and let
  \[
  A=A_S:=\End_\T(S).
  \]
  The subset
  $\twosilt{S}{}\T$ of $\silt \T$ given by
  \[
  \twosilt{S}{}\T := \setP{M\in\silt\T}{M\in S\ast (S[1])}
  \]
  plays an important role in the sequel. 
  The notation $\twosilt{S}{}\T$ is justified by the following
  remark.
\end{setting}

\begin{remark}
  Let $A$ be a finite dimensional algebra and 
  $\T=K^{\mathrm{b}}(\proj A)$.
  Then $A$ is a silting object in $\T$. 
  In this case a silting complex $M$ belongs to $\twosilt{A}{}\T$
  if and only if $M$ is isomorphic to a complex concentrated in
  degrees $-1$ and $0$, \ie if $M$ is a \emph{two-term silting
    complex}. 
\end{remark}

Following \cite[Sec. 2]{aihara_silting_2012}, we consider the
subcategory of $\T$ given by
\[
\T^{\leqslant0}:=
\setP{M\in\T}{\Hom_\T(S,M[i])=0\text{ for all }i>0}.
\]
We need the following generating properties of silting objects.
Recall that a pair $(\X,\Y)$ of subcategories of $\T$ is called
a \emph{torsion pair in $\T$} if $\Hom_\T(\X,\Y)=0$ and
$\T=\X\ast\Y$. 

\begin{proposition}
  \label{prop:silting-generating}
  \cite[Prop. 2.23]{aihara_silting_2012}
  With the hypotheses of Setting \ref{set:silting}, we have the
  following:
  \begin{align*}
    \T =& 
    \bigcup_{\ell\geqslant 0} S[-\ell]\ast S[1-\ell]\ast\cdots\ast S[\ell],\\
    \T^{\leqslant 0} =&
    \bigcup_{\ell\geqslant 0} S\ast S[1]\ast\cdots\ast S[\ell],\\
    \lperp{(\T^{\leqslant 0})} =&
    \bigcup_{\ell>0} S[-\ell]\ast S[1-\ell]\ast\cdots\ast S[-1].
  \end{align*}
  Moreover, the pair 
  $(\lperp{(\T^{\leqslant0})},\T^{\leqslant 0})$ is a torsion
  pair in $\T$.
\end{proposition}

The following proposition describes $\twosilt{S}{}\T$ in terms
of the partial order in $\silt\T$.
It is shown in  \cite[Prop. 2.9]{aihara_tilting-connected_2013}
in the case when $\T=K^{\mathrm{b}}(\proj A)$ and $S=A$.

\begin{proposition}
  \label{prop:order}
  Let $M$ be an object of $\T$. 
  Then, $M\in S\ast S[1]$ if and only if 
  $S[1]\leqslant M\leqslant S$.
\end{proposition}

\begin{proof}
  Before starting the proof, let us make the following trivial
  observation: 
  Given two subcategories $\X$ and $\Y$ of $\T$, for any object
  $M$ of $\T$, we have that
  $\Hom_\T(\X\ast\Y,M)=0$ if and only if
  $\Hom_\T(\X,M)=0$ and $\Hom_\T(\Y,M)=0$.
  
  Now note that we have $M\leqslant S$ if and only if
  $\Hom_\T(S[-i],M)=0$ for each $i>0$, or equivalently by the
  above observation, 
  $\Hom_\T(S[-\ell]\ast\cdots\ast S[-1],M)=0$ for each $\ell>0$. 
  Thus, by Proposition \ref{prop:silting-generating} we have that
  \begin{equation}
    \label{eq:prop-order-1}
    M\leqslant S
    \quad\text{if and only if}\quad
    \Hom_\T(\lperp{(\T^{\leqslant0})},M)=0
  \end{equation}
  or equivalently, since
  $(\lperp{(\T^{\leqslant0})},\T^{\leqslant0})$ is a torsion pair
  by Proposition \ref{prop:silting-generating},
  $M\in\T^{\leqslant0}=(S*S[1])*\T^{\leqslant0}[2]$. 
  By a similar argument, we have that
  \begin{equation}
    \label{eq:prop-order-2}
    S[1]\leqslant M
    \quad\text{if and only if}\quad
    \Hom_\T(M,\T^{\leqslant0}[2])=0.
  \end{equation}
  Then it follows from \eqref{eq:prop-order-1} and
  \eqref{eq:prop-order-2} that $S[1]\leqslant M\leqslant S$ if and
  only if $M\in S*S[1]$.
\end{proof}

We need the following result:

\begin{proposition}
  \label{prop:iy}
  \cite[Prop. 6.2(3)]{iyama_mutation_2008}
  The functor
  \begin{equation}
    \label{eq:bar}
    \b{(-)}=\Hom_\T(S,-):S\ast S[1]\to \mod A  
  \end{equation}
induces an equivalence of categories
\begin{equation}
  \label{eq:bar-equivalence}
  \b{(-)}:\frac{S\ast S[1]}{[S[1]]}\longrightarrow \mod A.  
\end{equation}
where $[S[1]]$ is the ideal of $\T$ consisting of morphisms which
factor through $\add S[1]$.
\end{proposition}
\begin{proof}
  Take $\X=\add S$, $\Y=\add S[1]$ and $\Z=\add S$ in
  \cite[Prop. 6.2(3)]{iyama_mutation_2008}.  
\end{proof}

In view of Proposition \ref{prop:iy}, for every $M,N\in S\ast
S[1]$ we have a natural isomorphism 
\[
\frac{\Hom_\T(M,N)}{[S[1]](M,N)} \cong \Hom_A(\b{M},\b{N}).
\]

\begin{setting}
  \label{set:silting-2}
  From now on, we fix a presilting object $U$ in $\T$ contained in
  $S\ast S[1]$.
  For simplicity, we assume that $U$ has no non-zero direct
  summands in $\add S[1]$.
  We are interested in the subset of $\twosilt{S}{}\T$ given by
  \[
  \twosilt{S}{U}\T:=\setP{M\in\twosilt{S}{}\T}{U\in\add S}.
  \]
\end{setting}

The following theorem is similar to
\cite[Thm. 3.2]{adachi_tau-tilting_2012}.

\begin{theorem}
  \label{thm:twosilt}
  \cite[Thm. 4.5]{iyama_intermediate_2013}
  With the hypotheses of Setting \ref{set:silting},
  the functor \eqref{eq:bar} induces an order-preserving
  bijection 
  \[
  \b{(-)}:\twosilt{S}{}\T \longrightarrow
  \sttilt A
  \]
  which induces a bijection
  \[
  \b{(-)}:\twosilt{S}{U}\T \longrightarrow
  \sttilt_{\b{U}} A.
  \]
\end{theorem}

Silting reduction was introduced in \cite[Thm. 2.37]{aihara_silting_2012}
in a special case and \cite{iyama_silting_2013} in the general case.
We are interested in the following particular situation:

\begin{theorem}
  \label{thm:silting-red}
  \cite{iyama_silting_2013}
  Let $U$ be a presilting object in $\T$ contained in $S\ast
  S[1]$. Then the canonical functor 
  \begin{equation}
    \label{eq:can}
    \T \longrightarrow \U:=\frac{\T}{\thick(U)}
  \end{equation}
  induces an order-preserving bijection
  \[
  \siltred:\setP{M\in\silt \T}{U\in\add M} \longrightarrow
  \silt\U.
  \]
\end{theorem}

We need to consider the following analog of Bongartz completion
for presilting objects in $S\ast S[1]$, \cf
\cite[Sec. 5]{derksen_general_2009} and
\cite[Prop. 6.1]{wei_semi-tilting_2012}. 

\begin{defprop}
  \cite[Prop. 2.16]{aihara_tilting-connected_2013}
  \label{def:bongartz-silting}
  Let $f:U'\to S[1]$ be a minimal right $(\add U)$-approximation
  of $S[1]$ in $\T$ and consider a triangle 
  \begin{equation}
    \label{eq:bongartz}
    S\to X_U \to U' \xto{f} S[1].
  \end{equation}
  Then $T_U:=X_U\oplus U$ is in $\twosilt{S}{}\T$ and moreover
  $T_U$ has no non-zero direct summands in $\add S[1]$. 
  We call $T_U$ the \emph{Bongartz completion of $U$ in 
    $S\ast S[1]$}.
\end{defprop}
\begin{proof}
  It is shown in \cite[Prop. 2.16]{aihara_tilting-connected_2013} 
  that $T_U$ is a silting object in $\T$. 
  Moreover, since $\Hom_\T(S,S[1])=0$ we have 
  \[
  T_U\in S\ast(S\ast S[1])=(S\ast S)\ast S[1] = S\ast S[1],
  \]
  hence $T_U\in\twosilt{S}{} \T$.
  Finally, since $\Hom_A(S,S[1])=0$ and $U$ has no non-zero direct 
  summands in $\add S[1]$, it follow from the triangle 
  \eqref{eq:bongartz} that $T_U$ has no non-zero direct summands
  in $\add S[1]$.
\end{proof}

We recall that by Proposition \ref{prop:sttiltUA} we have that
$\sttilt_{\b{U}} A$ equals the interval 
\[
\setP{M\in\sttilt A}{P(\Fac\b{U})\leqslant M\leqslant
  T_{\b{U}}}\subseteq\sttilt A;
\]
hence $T_{\b{U}}$ is the unique maximal element in
$\sttilt_{\b{U}} A$.
The following proposition relates the Bongartz completion $T_U$ of
$U$ in $S\ast S[1]$ with the Bongartz completion $T_{\b{U}}$ of
$\b{U}$ in $\mod A$. 

\begin{proposition}
  \label{prop:bongartz-silting}
  \begin{enumerate}
  \item $T_U$ is the unique maximal element in $\twosilt{S}{U}\T$.
  \item $\b{T_U} \cong T_{\b{U}}$.
  \end{enumerate}
\end{proposition}
\begin{proof}
  First, note that (b) follows easily from part (a).
  Indeed, since $T_U$ has no non-zero direct summands in $\add
  S[1]$, see Definition-Proposition \ref{def:bongartz-silting},
  we have that $|A|=|S|=|T_U|=|\b{T_U}|$.
  By Theorem \ref{thm:twosilt} we have that $\b{T_U}$ is a
  $\tau$-tilting $A$-module.
  Since $P(\lperp{(\tau\b{U})})$ is the unique maximal element in
  $\sttilt_{\b{U}} A$, to show part (b), \ie that 
  $\b{T_U}\cong P(\lperp{(\tau\b{U})})$, we only need to show that
  $T_U$ is the unique maximal element in $\twosilt{S}{U}\T$.
  
  For this, let $M\in\twosilt{S}{U}\T$ and fix $i>0$. 
  Applying the functor $\Hom_\T(-,M[i])$ to \eqref{eq:bongartz} we
  obtain an exact sequence 
  \[
  \Hom_\T(U', M[i]) \to \Hom_\T(X_U,M[i]) \to
  \Hom_\T(S,M[i]).
  \]
  Now, since $M$ is silting and $U'\in\add M$ we have that
  $\Hom_\T(U',M[i])=0$ for each $i>0$.
  On the other hand, since $M\in S\ast S[1]$, by Proposition
  \ref{prop:order} we have $\Hom_\T(S,M[i])=0$ for each $i>0$. 
  Thus he have $\Hom_\T(X_U,M[i])=0$ for each $i>0$.
  Since $T_U=X_U\oplus U$ we have $\Hom_\T(T_U,M[i])=0$ for
  each $i>0$, hence $M\leqslant T_U$. The claim follows.
\end{proof}

From this we can deduce the following result:

\begin{proposition}
  \label{prop:twosilt}
  Let $T_U\in\twosilt{S}{} \T$ be the Bongartz completion of $U$
  in $S\ast S[1]$.
  Then $T_U\cong S$ in $\U$ and the canonical functor
  \eqref{eq:can} induces an order preserving map 
  \[
  \siltred:\twosilt{S}{U}\T \longrightarrow
  \twosilt{T_U}{} \U.
  \]
\end{proposition}
\begin{proof}
  By \eqref{eq:bongartz} we have that
  $S\cong T_U$ in $\U=\lperp{(\tau U)}\cap U^\perp$, hence the canonical functor $\T\to 
  \U$ restricts to a functor 
  $S\ast S[1]\to T_U\ast T_U[1]\subset \U$. 
  The claim now follows from Theorem \ref{thm:silting-red}.
\end{proof}

We are ready to state the main theorem of this section. 
We keep the notation of the above discussion.

\begin{theorem}
  \label{thm:silting-compatibility}
  With the hypotheses of Settings \ref{set:silting} and
  \ref{set:silting-2}, we have the following:
  \begin{enumerate}
  \item The algebras $\End_\U(T_U)$ and 
    $C = \End_A(\b{T_U})/\langle e_{\b{U}}\rangle$ are isomorphic, 
    where $e_{\b{U}}$ is the idempotent corresponding to the
    projective $\End_A(\b{T_U})$-module $\Hom_A(\b{T_U},\b{U})$.
  \item We have a commutative diagram
  \begin{center}
    \begin{tikzcd}[column sep=huge]
      \twosilt{S}{U}\T \arrow{r}{\Hom_\T(S,-)}
      \dar[swap]{\siltred} &
      \sttilt_{\b{U}} A \dar{\taured} \\
      \twosilt{T_U}{} \U
      \arrow[swap]{r}{\Hom_\U(T_U,-)} &
      \sttilt C
    \end{tikzcd}   
  \end{center}
  in which each arrow is a bijection. 
  The vertical maps are given in Proposition \ref{prop:twosilt}
  and Theorem \ref{thm:tau-tilting-reduction} respectively. 
  \end{enumerate}
\end{theorem}

We begin by proving part (a) of Theorem
\ref{thm:silting-compatibility}. 
For this we need the following technical result.

Let $\UU$ be the subcategory of $\T$ given by
\[
\UU:=\setP{M\in\T}{\Hom_\T(M,U[i])=0
  \text{ and }
  \Hom_\T(U,M[i])=0
  \text{ for each }i>0}.
\]
Note that if $M$ is an object in $\silt_U\T$ then $M\in\UU$.
The following theorem allows us to realize $\U$ as a subfactor
category of $\T$.

\begin{theorem}
  \label{lemma:hom-verdier}
  \cite{iyama_silting_2013}
  The composition of canonical functors $\UU\to\T\to\U$ induces an
  equivalence 
  \[
  \frac{\UU}{[U]} \cong \U
  \]
  of additive categories.
  In particular, for every $M$ in $\UU$ there is a natural
  isomorphism
  \[
  \frac{\Hom_\T(T_U,M)}{[U](T_U,M)}
  \cong
  \Hom_\U(T_U,M)
  \]
\end{theorem}

Now we can prove the following lemma:

\begin{lemma}
  \label{lemma:well-defined-silting}
  For each $M$ in $\UU$ we have a functorial isomorphism
  \[
  \frac{\Hom_A(\b{T_U},\b{M})}{[\b{U}](\b{T_U},\b{M})}
  \cong 
  \Hom_\U(T_U,M).
  \]
\end{lemma}
\begin{proof}
  By \eqref{eq:bar-equivalence} we have the following functorial
  isomorphism 
  \[
  \frac{\Hom_A(\b{T_U},\b{M})}{[\b{U}](\b{T_U},\b{M})} 
  \cong
  \frac{\frac{\Hom_\C(T_U,M)}{[S[1]](T_U,M)]}}
  {\frac{[U](T_U,M)+[S[1]](T_U,M)}{[S[1]](T_U,M)}}.
  \]
  We claim that $[S[1]](T_U,M) \subseteq [U](T_U,M)$, or
  equivalently
  \[
  [S[1]](X_U,M)\subseteq [U](T_U,M)
  \]
  since $T_U=X\oplus U$. 
  Apply the contravariant functor $\Hom_\T(-,S[1])$ to the
  triangle \eqref{eq:bongartz} to obtain an exact sequence
  \[
  \Hom_\T(U',S[1]) \to \Hom_\T(X_U,S[1]) \to \Hom_\T(S,S[1])=0.
  \]
  Hence every morphism $X_U\to S[1]$ factors through $U'$, 
  and we have $[S[1]](X_U,M) \subseteq [U](X_U,M)$.
  Thus by Theorem \ref{lemma:hom-verdier} we have isomorphisms   
  \[
  \frac{\Hom_A(\b{T_U},\b{M})}{[\b{U}](\b{T_U},\b{M})} 
  \cong
  \frac{\Hom_\C(T_U,M)}{[U](T_U,M)}
  \cong
  \Hom_\U(T_U,M),
  \]
  which shows the assertion.
\end{proof}

Now part (a) of Theorem \ref{thm:silting-compatibility} follows by
putting $M=T_U$ in Lemma \ref{lemma:well-defined-silting}. 
In the remainder we prove Theorem
\ref{thm:silting-compatibility}(b).

For $X\in\mod A$ we denote by $0\to \t X\to X\to \f X\to 0$ the
canonical sequence of $X$ with respect to the torsion pair
$(\Fac\b{U},\b{U}^\perp)$ in $\mod A$.

\begin{proposition}
  \label{prop:E-C-modules}
  For each $M$ in $\UU$ there is an isomorphism of $C$-modules
  \[
  \Hom_A(\b{T_U},\f\b{M}) \cong \Hom_\U(T_U,M)
  \]
\end{proposition}
\begin{proof}
  By Lemma \ref{lemma:well-defined-silting} it is sufficient to
  show that 
  \[
  \frac{\Hom_A(\b{T_U},\b{M})}{[\b{U}](\b{T_U},\b{M})} \cong
  \Hom_A(\b{T_U},\f\b{M}).
  \]
  Apply the functor $\Hom_A(\b{T_U},-)$ to the canonical sequence
  \[
  0\to \t\b{M}\xto{i} \b{M}\to \f\b{M}\to 0
  \]
  to obtain an exact  sequence
  \[
  0\to \Hom_A(\b{T_U},\t\b{M}) \xto{i\circ-}
  \Hom_A(\b{T_U},\b{M})\to 
  \Hom_A(\b{T_U},\f\b{M}) \to \Ext_A^1(\b{T_U},\t\b{M})=0,
  \]
  since $\b{T_U}$ is $\Ext$-projective in $\lperp{(\tau \b{U})}$
  by Proposition \ref{prop:bongartz-silting}
  and $\t\b{M}$ is in $\Fac\b{U}\subseteq\lperp{(\tau\b{U})}$. 
  Thus 
  \[
  \Hom_A(\b{T_U},\f\b{M}) \cong
  \frac{\Hom_A(\b{T_U},\b{M})}{i(\Hom_A(\b{T_U},\t\b{M}))}.
  \]
  Thus we only have to show the equality
  $\Hom_A(\b{T_U},\t\b{M})=i([U](\b{T_U},\t\b{M}))$. 
  First, $\Hom_A(\b{T_U},\t\b{M})\subseteq
  i([U](\b{T_U},\t\b{M}))$
  since $i$ is a right $(\Fac U)$-approximation of $\b{M}$.
  Next we show the reverse inclusion. 
  It is enough to show that every map
  $\b{T_U}\to \t\b{M}$ factors through $\add\b{U}$.
  Let $r:\b{U'}\to \t\b{M}$ be a right $(\add\b{U})$-approximation.
  Since $t\b{M}\in\Fac\b{U}$ we have a short exact sequence
  \[
  0\to \b{K}\to \b{U'}\xto{r} \t\b{M} \to 0.
  \]
  Moreover, by Lemma \ref{lemma:wakamatsu}, we have
  $\b{K}\in\lperp{(\tau \b{U})}$. 
  Apply the functor $\Hom_A(\b{T_U},-)$ to the above sequence to
  obtain an exact sequence
  \[
  \Hom_A(\b{T_U},\b{U'}) \to \Hom_A(\b{T_U},\t\b{M}) \to
  \Ext_A^1(\b{T_U},\b{K})=0.
  \]
  Thus the assertion follows.
\end{proof}

We are ready to prove Theorem \ref{thm:silting-compatibility}(b).

\begin{proof}[Proof of Theorem \ref{thm:silting-compatibility}(b)]
  Let $M\in\twosilt{S}{U}\T$. 
  Then $M\in\UU$. We only need to show that $\Hom_\U(T_U,M)$
  coincides with the $\tau$-tilting reduction of 
  $\b{M}\in\sttilt A$ with respect to $\b{U}$, which is given by
  $\Hom_A(\b{T_U},\f\b{M})$,
  see Theorem \ref{thm:tau-tilting-reduction}. 
  This is shown in Proposition \ref{prop:E-C-modules}.
\end{proof}

\begin{corollary}
  The map
  \[
  \siltred:\setP{M\in\twosilt{S}{} \T}{U\in\add M} \longrightarrow
  \twosilt{T_U}{} \U.
  \]
  is bijective.
\end{corollary}

\subsection{Calabi-Yau reduction}

Let $\C$ be a Krull-Schmidt 2-Calabi-Yau triangulated category.
Thus $\C$ is $k$-linear, $\Hom$-finite and there is a bifunctorial
isomorphism
\[
\Hom_\C(M,N)\cong D\Hom_\C(N,M[2])
\]
for every $M,N\in\C$, where $D = \Hom_k(-,k)$ is the usual
$k$-duality.
Recall that an object $T$ in $\C$ is called \emph{cluster-tilting}
if
\[
\add T = \setP{X\in\C}{\Hom_\C(X,T[1])=0}.
\] 
We denote by $\ctilt\C$ the set of isomorphism classes of all
basic cluster-tilting objects in $\C$.

\begin{setting}
  \label{set:2cy}
  We fix a Krull-Schmidt 2-Calabi-Yau triangulated category $\C$
  with a cluster-tilting object $T$.
  Also, we let
  \[
  A = \End_\C(T).
  \]
  The algebra $A$ is said to be \emph{2-Calabi-Yau tilted}.
\end{setting}

Note that the functor
\begin{equation}
  \label{eq:bar-2cy}
  \b{(-)}=\Hom_\C(T,-):\C \longrightarrow \mod A  
\end{equation}
induces an equivalence of categories
\begin{equation}
  \label{eq:bar-equivalence-2cy}
\b{(-)}: \frac{\C}{[T[1]]} \longrightarrow \mod A 
\end{equation}
where $[T[1]]$ is the ideal of $\C$ consisting of morphisms which
factor through $\add T[1]$, 
\cite[Prop. 2(c)]{keller_cluster-tilted_2007}. 
Thus for every $M,N\in\C$ we have a natural isomorphism
\[
\Hom_A(\b{M},\b{N}) \cong \frac{\Hom_\C(M,N)}{[T[1]](M,N)}.
\]
We have the following result:

\begin{theorem}
  \label{thm:cttilt-sttilt}
  \cite[Thm. 4.1]{adachi_tau-tilting_2012}
  With the hypotheses of Setting \ref{set:2cy}, the functor
  \eqref{eq:bar-2cy} sends rigid object in $\C$ to $\tau$-rigid
  objects in $\mod A$, and induces a bijection
  \[
  \b{(-)}:\ctilt \C \longrightarrow \sttilt A
  \]
  which induces a bijection
  \[
  \b{(-)}:\ctilt_U \C \longrightarrow \sttilt_U A.
  \]
\end{theorem}

\begin{setting}
  \label{set:2cy-2}
  From now on we fix a rigid object $U$ in $\C$, \ie
  $\Hom_\C(U,U[1])=0$.
  To simplify the exposition, we assume that $U$ has no non-zero direct
  summands in $\add T[1]$ although our results remain true in this case.
  We are interested in the subset of $\ctilt\C$ given by 
  \[
  \ctilt_U\C := \setP{M\in\ctilt\C}{U\in\add M}.
  \]
\end{setting}

Calabi-Yau reduction was introduced in \cite[Thm. 4.9]{iyama_mutation_2008}. 
We are interested in the following particular case: 

\begin{theorem}
  \label{thm:2cy-reduction}
  \cite[Sec. 4]{iyama_mutation_2008}
  The category $\U$ has the structure of a 2-Calabi-Yau
  triangulated category and the canonical functor 
  \[ 
  \lperp{(U[1])} \longrightarrow \U:=\frac{\lperp{(U[1])}}{[U]}
  \]
  induces a bijection
  \[
  \ctred:\ctilt_U \C \longrightarrow \ctilt\U. 
  \]
\end{theorem}
 
We need to consider the following analog of Bongartz completion
for rigid objects in $\C$.

\begin{defprop}
  Let $f:U'\to T[1]$ be a  minimal right $(\add U)$-approximation
  of $T[1]$ in $\C$ and consider a triangle
  \begin{equation}
    \label{eq:bongartz-2cy}
    T\xto{h} X_U \xto{g} U' \xto{f} T[1].    
  \end{equation}
  Then $T_U:= X_U\oplus U$ is cluster-tilting in $\C$ and moreover
  $T_U$ has no non-zero direct summands in $\add T[1]$. 
  We call $T_U$ the \emph{Bongartz completion of $U$ in $\C$
    with respect to $T$}.
\end{defprop}
\begin{proof}
  (i) First we show that $T_U$ is rigid. Our argument is a
  triangulated version of the proof of 
  \cite[Prop. 5.1]{geis_rigid_2006}.
  We give the proof for the convenience of the reader.

  Apply the functor $\Hom_\C(U,-)$ to the triangle
  \eqref{eq:bongartz-2cy} to obtain an exact sequence
  \[
  \Hom_\C(U,U') \xto{f\circ-} \Hom_\C(U,T[1]) \to
  \Hom_\C(U,X_U[1]) \to \Hom_\C(U,U'[1])=0.
  \]
  But $f:U'\to T[1]$ is a right $(\add U)$-approximation of
  $T[1]$,
  thus $f\circ-$ is an epimorphism and we have
  $\Hom_\C(U,X_U[1])=0$ and, by the 2-Calabi-Yau property,
  $\Hom_\C(X_U,U[1])=0$. 
  It remains to show that $\Hom_\C(X_U,X_U[1])=0$. 
  For this let $a:X_U\to X_U[1]$ be an arbitrary morphism. 
  Since $\Hom_\C(X_U,U'[1])=0$ there exists a morphism
  $h:X_U\to T[1]$ such that the following diagram commutes:
  \begin{center}
    \begin{tikzcd}
      {}& T \rar & X_U \rar{g}\dar{a}\dlar[swap,dashed]{h} & U' \rar
      & T[1] \\  
      U' \rar & T[1] \rar{h[1]} & X_U[1] \rar & U'[1]
    \end{tikzcd}
  \end{center}
  Now, since $\Hom_\C(T,T[1])=0$ there exist a morphism $c:U'\to
  T[1]$ such that $h=cg$. 
  Then we have
  \[
  h[1](cg) = (h[1]\circ c)g =0,
  \]
  since $(h[1]\circ c)\in\Hom_\C(U',X_U[1])=0$. 
  Hence $\Hom_\C(X_U,X_U[1])=0$ as required. 
  Thus we have shown that $T_U$ is rigid.

  (ii) Now we show that $T_U$ is cluster-tilting in $\C$. 
  By the bijection in Theorem \ref{thm:cttilt-sttilt}, we
  only need to show that $\b{T_U}$ is a support $\tau$-tilting
  $A$-module. 
  Since $T_U$ is rigid, we have that $\b{T_U}$ is $\tau$-rigid
  (see Theorem \ref{thm:cttilt-sttilt}).
  Apply the functor \eqref{eq:bar-2cy} to the triangle 
  \eqref{eq:bongartz-2cy} to obtain an exact sequence
  \[
  A \xto{\bar{g}} \b{X_U} \to \b{U'} \to 0.
  \]
  We claim that $\b{g}$ is a left $(\add\b{T_U})$-approximation of
  $A$. 
  In fact, since we have $\Hom_\C(U[-1],T_U)=0$, for every morphism
  $h:T\to T_U$ in $\C$ we obtain a commutative diagram
  \begin{center}
    \begin{tikzcd}
      U[-1] \rar\drar[swap,dashed]{0} & T\rar{g}\dar{h} & X_U\rar
      \dlar[dashed] & U \\ 
      & T_U
    \end{tikzcd}    
  \end{center}
  Thus $g$ is a left $(\add T_U)$-approximation of $T$ and then by
  the equivalence \eqref{eq:bar-equivalence-2cy} we have that
  $\bar{g}$ is a left $(\add \b{T_U})$-approximation of $A$.
  Then by Proposition \ref{prop:sttilt-sequence} we have that
  $\b{T_U}$ is a support $\tau$-tilting $A$-module.

  Finally, since $\Hom_A(T,T[1])=0$ and $U$ has no non-zero direct
  summands in $\add T[1]$, it follow from the triangle
  \eqref{eq:bongartz-2cy} that $T_U$ has no non-zero direct
  summands in $\add T[1]$. 
\end{proof}

The following proposition relates the Bongartz completion $T_U$ of
$U$ in $\C$ with respect to $T$ with the Bongartz completion
$\b{T_U}$ of $\b{U}$ in $\mod A$. 
Recall that $T_{\b{U}}$ is the unique maximal element in
$\sttilt_{\b{U}} A$. 

\begin{proposition}
  \label{prop:bongartz-silting}
  We have $\b{T_U} \cong T_{\b{U}}$.
\end{proposition}
\begin{proof}
  By Proposition \ref{prop:sttiltUA}, $P(\lperp{(\tau \b{U})})$ is
  the unique maximal element in $\sttilt_{\b{U}} A$.
  Hence to show that $\b{T_U}\cong P(\lperp{(\tau \b{U})})$ we
  only need to show that if $\b{M}$ is a support $\tau$-tilting
  $A$-module such that $\b{U}\in\add \b{M}$, then 
  $\b{M}\in\Fac\b{T_U}$. 
  By definition, this is equivalent to show that there exists an
  exact sequence of $A$-modules
  \[
  \Hom_\C(T,X)\to \Hom_\C(T,M)\to 0
  \]
  with $X\in\add T_U$. 

  Let $f:T'\to M$ be a right $(\add T)$-approximation
  of $M$. By the definition of $X_U$, there exist a triangle 
  \[
  T'\xto{g}X\to U''\to T[1]
  \]
  where $X\in\add X_U$ and $U''\in\add U$. Since
  $\Hom_\C(U'[-1],M)=D\Hom_\C(M,U'[1])=0$
  by the 2-Calabi-Yau property of $\C$, the following diagram commutes:
  \[
  \begin{tikzcd}
    U''[-1]\rar\drar{0}&T'\dar{f}\rar{g}&X\rar\dlar[dotted]{h}& U''\\
    &M
  \end{tikzcd}
  \]
  It follows that there exist a morphism $h:X\to M$ such that
  $f=hg$. Now it is easy to see that the sequence 
  \[
    \Hom_\C(T,X)\xto{h\cdot?}\Hom_\C(T,M)\to0
  \]
  is exact. Indeed, let $ u:T\to M$. since $f$ is a right $(\add
  T)$-approximation of $M$,
  there exist a morphism $v:T\to T'$ such that $u=fv$. It follows that
  \[
  u=fv=(hg)v=h(gv).
  \]
  This shows that $u$ factors through $h$, which is what we needed to
  show.
\end{proof}

We are ready to state the main theorem of this section. 
We keep the notation of the above discussion.

\begin{theorem}
  \label{thm:2-cy-compatibility}
  With the hypotheses of Settings \ref{set:2cy} and
  \ref{set:2cy-2}, we have the following:
  \begin{enumerate}
  \item The algebras $\End_\U(T_U)$ and 
    $C = \End_A(\b{T_U})/\langle e_{\b{U}}\rangle$ are isomorphic,
    where $e_{\b{U}}$ is the idempotent corresponding to the
    projective $\End_A(\b{T_U})$-module $\Hom_A(\b{T_U},\b{U})$.
  \item We have a commutative diagram
    \begin{center}
      \begin{tikzcd}[column sep=huge]
        \ctilt_U\C \arrow{r}{\Hom_\C(T,-)}
        \dar[swap]{\ctred} &
        \sttilt_{\b{U}} A \dar{\taured} \\
        \ctilt\U
        \arrow[swap]{r}{\Hom_\U(T_U,-)} &
        \sttilt C
      \end{tikzcd}    
    \end{center}
    in which each arrow is a bijection.
    The vertical maps are given in Theorem \ref{thm:2cy-reduction}
    and  Theorem \ref{thm:tau-tilting-reduction}.
  \end{enumerate}
\end{theorem}

We begin with the proof of part (a) of
Theorem \ref{thm:2-cy-compatibility}.

\begin{lemma}
  \label{lemma:well-defined}
  For each $M$ in $\lperp{(U[1])}$ we have an isomorphism of
  vector spaces 
  \[
  \frac{\Hom_A(\b{T_U},\b{M})}{[\b{U}](\b{T_U},\b{M})}
  \cong 
  \Hom_\U(T_U,M).
  \]
\end{lemma}

\begin{proof}
  By the equivalence \eqref{eq:bar-equivalence-2cy}, 
  we have the following isomorphism
  \[ 
  \frac{\Hom_A(\b{T_U},\b{M})}{[\b{U}](\b{T_U},\b{M})} 
  \cong
  \frac{\frac{\Hom_\C(T_U,M)}{[T[1]](T_U,M)]}}
  {\frac{[U](T_U,M)+[T[1]](T_U,M)}{[S[1]](T_U,M)}}.
  \]
  We claim that $[T[1]](T_U,M) \subseteq [U](T_U,M)$, or
  equivalently $[T[1]](X_U,M) \subseteq [U](X_U,M)$ since
  $T_U=X_U\oplus U$. 
  Apply the contravariant functor 
  $\Hom_\C(-,T[1])$ to the triangle \eqref{eq:bongartz-2cy} to
  obtain an exact sequence
  \[
  \Hom_\C(U',T[1]) \to \Hom_\C(X_U,T[1]) \to \Hom_\C(T,T[1])=0.
  \]
  Hence every morphism $X_U\to T[1]$ factors through $U'$, and
  we have $[T[1]](X_U,M) \subseteq [U](X_U,M)$. 
  Thus we have the required isomorphisms
  \[
  \frac{\Hom_A(\b{T_U},\b{M})}{[\b{U}](\b{T_U},\b{M})} 
  \cong
  \frac{\Hom_\C(T_U,M)}{[U](T_U,M)}
  \cong
  \Hom_\U(T_U,M),
  \]
  so the assertion follows.
\end{proof}

Now part (a) of Theorem \ref{thm:2-cy-compatibility} follows by
putting $M=T_U$ in Lemma \ref{lemma:well-defined}. 
In the remainder we prove Theorem
\ref{thm:2-cy-compatibility}(b).
 
For $X\in\mod A$ we denote by $0\to \t X\to X\to \f X\to 0$ the
canonical sequence of $X$ with respect to the torsion pair
$(\Fac\b{U},\b{U}^\perp)$ in $\mod A$.

\begin{proposition}
  \label{prop:E-C-modules-cy}
  For each $M$ in $\lperp{(U[1])}$ there is an isomorphism of
  $C$-modules
  \[
  \Hom_\U(T_U,M) \cong \Hom_A(\b{T_U},\f\b{M}).
  \]
\end{proposition}
\begin{proof}
  By Lemma \ref{lemma:well-defined} it is enough to show that 
  \[
  \frac{\Hom_A(\b{T_U},\b{M})}{[\b{U}](\b{T_U},\b{M})}
  \cong
  \Hom_A(\b{T_U},\f\b{M}).
  \]
  We can proceed exactly as in the proof of Proposition
  \ref{prop:E-C-modules}.
\end{proof}

We are ready to give the prove Theorem
\ref{thm:2-cy-compatibility}(b).

\begin{proof}[Proof of Theorem \ref{thm:2-cy-compatibility}(b)]
  Let $M\in\ctilt\C$ be such that $U\in\add M$. 
  Since $\tau$-tilting reduction of $\b{M}$ is given by
  $F(\f\b{M})=\Hom_A(T_{\b{U}},\f M)$, see Theorem
  \ref{thm:tau-tilting-reduction}, we only need to show that
  \[
  \Hom_\U(T_U,M) = \Hom_A(T_{\b{U}},\f M).
  \]
  But this is precisely the statement of Proposition
  \ref{prop:E-C-modules-cy} since $\b{T_U} \cong T_{\b{U}}$ by
  Proposition \ref{prop:bongartz-silting}.
\end{proof}

We conclude this section with an example illustrating the
bijections in Theorem \ref{thm:2-cy-compatibility}.

\begin{example}
  Let $\C$ be the cluster category of type $D_4$. Recall that $\C$
  is a 2-Calabi-Yau triangulated category, 
  see \cite{buan_tilting_2006}. 
  The Auslander-Reiten quiver of $\C$ is the following, where
  the dashed edges are to be identified to form a cylinder, 
  see \cite[6.4]{schiffler_geometric_2008}:
  \begin{center}
    $\C:\quad$ \begin{tikzpicture}[commutative diagrams/every diagram]
  \node (P1) at (mesh cs:u=0,v=0,r=0.7) {$\bullet$};
  \node (S2) at (mesh cs:u=-1,v=0,r=0.7) {$\bullet$};
  \node[draw,rectangle] (S3) at (mesh cs:u=-1,v=-1,r=0.35) {$T_4$};
  \node (S4) at (mesh cs:u=0,v=-1,r=0.7) {$U[1]$};

  \node (M1) at (mesh cs:u=1,v=1,r=0.7) {$\bullet$};
  \node[draw,rectangle]  (M2) at (mesh cs:u=0,v=1,r=0.7) {$T_3$};
  \node (M3) at (mesh cs:u=1,v=1,r=0.35) {$\bullet$};
  \node[draw,rectangle] (M4) at (mesh cs:u=1,v=0,r=0.7) {$U$};

  \node (S1) at (mesh cs:u=2,v=2,r=0.7) {$\bullet$};
  \node (I2) at (mesh cs:u=1,v=2,r=0.7) {$\bullet$};
  \node[draw,rectangle]  (I3) at ($(mesh cs:u=1,v=1,r=0.7)+(mesh
  cs:u=1,v=1,r=0.35)$) {$T_2$}; 
  \node (I4) at (mesh cs:u=2,v=1,r=0.7) {$\bullet$};
  
  \node (TP1) at (mesh cs:u=3,v=3,r=0.7) {$\bullet$};
  \node (TS2) at (mesh cs:u=2,v=3,r=0.7) {$T_1$};
  \node (TS3) at ($(mesh cs:u=2,v=2,r=0.7)+(mesh
  cs:u=1,v=1,r=0.35)$) {$\bullet$}; 
  \node (TS4) at (mesh cs:u=3,v=2,r=0.7) {$\bullet$};

  \node (S2a) at (mesh cs:u=3,v=4,r=0.7) {$\bullet$};
  \node[draw,rectangle]  (S3a) at ($(mesh cs:u=3,v=3,r=0.7)+(mesh
  cs:u=1,v=1,r=0.35)$) {$T_4$}; 
  \node (S4a) at (mesh cs:u=4,v=3,r=0.7) {$U[1]$};
 
  \draw[rounded corners=4pt] ($(S2)+(-0.4,0.4)$) --
  ($(I2)+(0.4,0.4)$) --  ($(I3)+(0.4,-0.4)$) --
  ($(I3)+(-0.4,-0.4)$) --
  ($(M4)+(0.4,-0.4)$) --  ($(M4)+(-0.4,-0.4)$) --
  ($(S3)+(0.4,-0.4)$) --
  ($(S3)+(-0.4,-0.4)$) -- cycle;

  \draw[rounded corners=4pt] ($(S2a)+(-0.4,0.4)$) --
  ($(S2a)+(0.4,0.4)$) --
  ($(S3a)+(0.4,-0.4)$) --
  ($(S3a)+(-0.4,-0.4)$) -- cycle;

  \draw[rounded corners=4pt] ($(TS4)+(-0.4,0.4)$) --
  ($(TS4)+(0.4,0.4)$) --
  ($(TS4)+(0.4,-0.4)$) --
  ($(TS4)+(-0.4,-0.4)$) -- cycle;

  \path[commutative diagrams/.cd, every arrow]
  (S2) edge (P1)
  (S3) edge (P1)
  (S4) edge (P1)
  (P1) edge (M2)
  (P1) edge (M3)
  (P1) edge (M4)
  (M2) edge (M1)
  (M3) edge (M1)
  (M4) edge (M1)
  (M1) edge (I2)
  (M1) edge (I3)
  (M1) edge (I4)
  (I2) edge (S1)
  (I3) edge (S1)
  (I4) edge (S1)
  (S1) edge (TS2)
  (S1) edge (TS3)
  (S1) edge (TS4)
  (TS2) edge (TP1)
  (TS3) edge (TP1)
  (TS4) edge (TP1)
  (TP1) edge (S2a)
  (TP1) edge (S3a)
  (TP1) edge (S4a)
  (S2)  edge[commutative diagrams/path, dashed] (S3)
  (S3)  edge[commutative diagrams/path, dashed] (S4)
  (S2a)  edge[commutative diagrams/path, dashed] (S3a)
  (S3a)  edge[commutative diagrams/path, dashed] (S4a);
\end{tikzpicture}

%%% Local Variables: 
%%% mode: latex
%%% TeX-master: "../master"
%%% End:     
  \end{center}
  We have chosen a cluster-tilting object $T=T_1\oplus T_2\oplus
  T_3\oplus T_4$ and a rigid indecomposable object $U$ in
  $\C$. 
  The Bongartz completion of $U$ with respect to $T$ is given
  by $T_U = U\oplus T_2\oplus T_3\oplus T_4$, and is indicated
  with squares. 
  Also, the ten indecomposable objects of the subcategory
  $\lperp{(U[1])}$ have  been encircled.

  On the other hand, let $A=\End_\C(T)$ and
  $\b{(-)}=\Hom_\C(T,-)$. Then $A$ is isomorphic to algebra given
  by the quiver
  \begin{center}
    $Q'\quad = \quad$
    \begin{tikzcd}
      2 \dar[swap]{x_2} & 1 \lar[swap]{x_1} \\
      3 \rar[swap]{x_3} & 4 \uar[swap]{x_4}
    \end{tikzcd}    
  \end{center}
  with relations $x_1x_2x_3=0$, $x_2x_3x_4=0$, $x_3x_4x_1=0$ and
  $x_4x_1x_2=0$. 
  Thus $A$ is isomorphic to the Jacobian algebra of the quiver
  with potential $(Q',x_1x_2x_3x_4)$. 
  The Auslander-Reiten quiver of $\mod A$ is the following: 
  \begin{center}
    $\mod A:\quad$ 
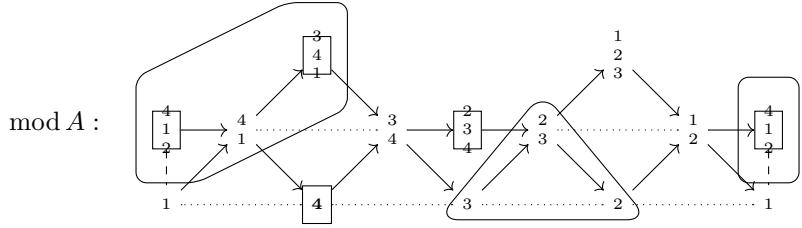
\begin{tikzpicture}[commutative diagrams/every diagram]
  \node (P1) at (mesh cs:u=0,v=0,r=0.7) {$\rep{4\\1}$};
  % \node (S2) at (mesh cs:u=-1,v=0,r=0.7) {$\bullet$};
  \node[draw,rectangle] (S3) at (mesh cs:u=-1,v=-1,r=0.35) {$\rep{4\\1\\2}$};
  \node (S4) at (mesh cs:u=0,v=-1,r=0.7) {$\rep{1}$};

  \node (M1) at (mesh cs:u=1,v=1,r=0.7) {$\rep{3\\4}$};
  \node[draw,rectangle] (M2) at (mesh cs:u=0,v=1,r=0.7) {$\rep{3\\4\\1}$};
  % \node (M3) at (mesh cs:u=1,v=1,r=0.35) {$\bullet$};
  \node[draw,rectangle] (M4) at (mesh cs:u=1,v=0,r=0.7) {$\rep{\mathbf{4}}$};

  \node (S1) at (mesh cs:u=2,v=2,r=0.7) {$\rep{2\\3}$};
  % \node (I2) at (mesh cs:u=1,v=2,r=0.7) {$\bullet$};
  \node[draw,rectangle] (I3) at ($(mesh cs:u=1,v=1,r=0.7)+(mesh
  cs:u=1,v=1,r=0.35)$) {$\rep{2\\3\\4}$}; 
  \node (I4) at (mesh cs:u=2,v=1,r=0.7) {$\rep{3}$};
  
\node (TP1) at (mesh cs:u=3,v=3,r=0.7) {$\rep{1\\2}$};
  \node (TS2) at (mesh cs:u=2,v=3,r=0.7) {$\rep{1\\2\\3}$};
  % \node (TS3) at ($(mesh cs:u=2,v=2,r=0.7)+(mesh
  % cs:u=1,v=1,r=0.35)$) {$\bullet$}; 
  \node (TS4) at (mesh cs:u=3,v=2,r=0.7) {$\rep{2}$};

  % \node (S2a) at (mesh cs:u=3,v=4,r=0.7) {$\bullet$};
  \node[draw,rectangle] (S3a) at ($(mesh cs:u=3,v=3,r=0.7)+(mesh
  cs:u=1,v=1,r=0.35)$) {$\rep{4\\1\\2}$}; 
  \node (S4a) at (mesh cs:u=4,v=3,r=0.7) {$\rep{1}$};

  \draw[rounded corners=4pt] ($(S3)+(-0.4,0.7)$) --
  ($(M2)+(-0.4,0.7)$) --  ($(M2)+(0.4,0.7)$) --
  ($(M2)+(0.4,-0.7)$) --  ($(S3)+(0.4,-0.7)$) --
  ($(S3)+(-0.4,-0.7)$) -- cycle;

  \draw[rounded corners=8pt] ($(I4)-(0.4,0.2)$) -- ($(S1)+(0,0.5)$) --
  ($(TS4)-(-0.4,0.2)$) -- cycle;

  \draw[rounded corners=4pt] ($(S3a)+(-0.4,0.7)$) --
  ($(S3a)+(0.4,0.7)$) --
  ($(S3a)+(0.4,-0.7)$) -- ($(S3a)+(-0.4,-0.7)$) -- cycle;

  \path[commutative diagrams/.cd, every arrow]
  % (S2) edge (P1)
  (S3) edge (P1)
  (S4) edge (P1)
  (P1) edge (M2)
  % (P1) edge (M3)
  (P1) edge (M4)
  (M2) edge (M1)
  % (M3) edge (M1)
  (M4) edge (M1)
  % (M1) edge (I2)
  (M1) edge (I3)
  (M1) edge (I4)
  % (I2) edge (S1)
  (I3) edge (S1)
  (I4) edge (S1)
  (S1) edge (TS2)
  % (S1) edge (TS3)
  (S1) edge (TS4)
  (TS2) edge (TP1)
  % (TS3) edge (TP1)
  (TS4) edge (TP1)
  % (TP1) edge (S2a)
  (TP1) edge (S3a)
  (TP1) edge (S4a)
  % (S2)  edge[commutative diagrams/path, dotted] (S3)
  (S3)  edge[commutative diagrams/path, dashed] (S4)
  % (S2a)  edge[commutative diagrams/path, dotted] (S3a)
  (S3a)  edge[commutative diagrams/path, dashed] (S4a)

  (M1)  edge[commutative diagrams/path, dotted] (P1)
  (TP1)  edge[commutative diagrams/path, dotted] (S1)

  (M4)  edge[commutative diagrams/path, dotted] (S4)
  (I4)  edge[commutative diagrams/path, dotted] (M4)
  (TS4)  edge[commutative diagrams/path, dotted] (I4)
  (S4a)  edge[commutative diagrams/path, dotted] (TS4);
\end{tikzpicture}

%%% Local Variables: 
%%% mode: latex
%%% TeX-master: "../master"
%%% End:     
  \end{center}
  We have indicated the indecomposable direct summands of the
  Bongartz completion $\b{T_U}$ of $\b{U}=S_4$, with rectangles. 
  The six indecomposable objects of the category
  $\lperp{(\tau \b{U})}\cap \b{U}^\perp$ are encircled. 
  Finally, let $C=\End_A(\b{T_U})/\langle e_{\b{U}}\rangle$.
  Not that $C$ is isomorphic by the algebra by the quiver 
  \begin{center}
    $Q\quad = \quad$
    \begin{tikzpicture}[commutative diagrams/every diagram]
      \node (S2) at (mesh cs:u=0,v=0,r=0.7){$2$};
      \node (P1) at (mesh cs:u=0,v=1,r=0.7){$1$};
      \node (S1) at (mesh cs:u=1,v=1,r=0.7){$3$};
      \path[commutative diagrams/.cd, every arrow, every
      label]
      (P1) edge node [swap]{$x_1$} (S2)
      (S2) edge node [swap]{$x_2$} (S1)
      (S1) edge node [swap]{$x_3$} (P1);
    \end{tikzpicture}        
  \end{center}
  with relations $x_1x_2=0$, $x_2x_3=0$ and $x_3x_1=0$, see
  Example \ref{ex:Nak3}. 
  Thus $C$ is isomorphic to the Jacobian algebra of the quiver
  with potential $(Q,x_1x_2x_3)$, see \cite{derksen_quivers_2008} 
  for definitions. 
  By Theorem \ref{thm:U-modC} we have that 
  $\lperp{(\tau\b{U})}\cap\b{U}^\perp$ is equivalent to 
  $\mod C$.
  The Auslander-Reiten quiver of $\mod C$ is the following, where
  each $C$-module is represented by its radical filtration:  
  \begin{center}
    $\mod C:\quad$ 
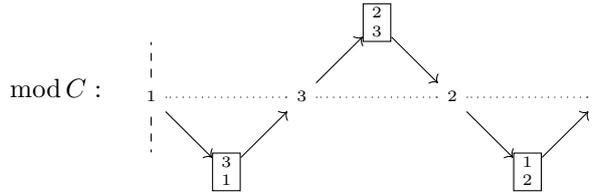
\begin{tikzpicture}[commutative diagrams/every diagram]
  \node (S1a) at (mesh cs:u=-1,v=-1,r=0.7) {$\rep{1}$}; 
  \node[draw,rectangle] (P3) at (mesh cs:u=0,v=-1,r=0.7)
  {$\rep{3\\1}$}; 

  \node (S3) at (mesh cs:u=0,v=0,r=0.7) {$\rep{3}$};

  \node[draw,rectangle] (P2) at (mesh cs:u=0,v=1,r=0.7)
  {$\rep{2\\3}$}; 

  \node (S2) at (mesh cs:u=1,v=1,r=0.7) {$\rep{2}$};

  \node[draw,rectangle] (P1) at (mesh cs:u=2,v=1,r=0.7)
  {$\rep{1\\2}$}; 

  \node (S1) at (mesh cs:u=2,v=2,r=0.7) {$\rep{1}$};

  \node (X1) at ($(mesh cs:u=-1,v=-1,r=0.7)+(mesh
  cs:u=-1,v=1,r=0.35)$) {}; 
  \node (X2) at ($(mesh cs:u=-1,v=-1,r=0.7)+(mesh
  cs:u=1,v=-1,r=0.35)$) {}; 

  \node (Y1) at ($(mesh cs:u=2,v=2,r=0.7)+(mesh
  cs:u=-1,v=1,r=0.35)$) {}; 
  \node (Y2) at ($(mesh cs:u=2,v=2,r=0.7)+(mesh
  cs:u=1,v=-1,r=0.35)$) {}; 

  \path[commutative diagrams/.cd, every arrow]
  (S1a) edge (P3)
  (P3) edge (S3)
  (S3) edge (P2)
  (P2) edge (S2)
  (S2) edge (P1)
  (P1) edge (S1)
  % (S1) edge (P3a)
  (S1a)  edge[commutative diagrams/path, dashed] (X1)
  (S1a)  edge[commutative diagrams/path, dashed] (X2)
  (S1)  edge[commutative diagrams/path, dashed] (Y1)
  (S1)  edge[commutative diagrams/path, dashed] (Y2)
  (S1)  edge[commutative diagrams/path, dotted] (S2)
  (S2)  edge[commutative diagrams/path, dotted] (S3)
  (S3)  edge[commutative diagrams/path, dotted] (S1a)
  ;
  % (S3)  edge[commutative diagrams/path, dashed] (S2)
  % (S1a)  edge[commutative diagrams/path, dashed] (S3);

\end{tikzpicture}

%%% Local Variables: 
%%% mode: latex
%%% TeX-master: "../master"
%%% End:     
  \end{center}

  On the other hand, let  $\U = \lperp{(U[1])}/[U]$. 
  The Auslander-Reiten quiver of $\U$ is the following, note that
  the dashed edges are to be identified to form a M\"obius strip:
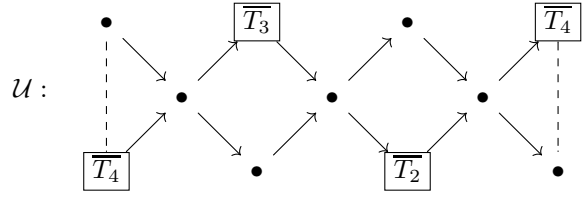
  \begin{center}
    $\U:\quad$ \begin{tikzpicture}[commutative diagrams/every diagram]
  \node (P1) at (mesh cs:u=0,v=0,r=0.7) {$\bullet$};
  \node (S2) at (mesh cs:u=-1,v=0,r=0.7) {$\bullet$};
  % \node (S3) at (mesh cs:u=-1,v=-1,r=0.35) {$T_1$};
  \node[draw,rectangle] (S4) at (mesh cs:u=0,v=-1,r=0.7) {$\b{T_4}$};

  \node (M1) at (mesh cs:u=1,v=1,r=0.7) {$\bullet$};
  \node[draw,rectangle] (M2) at (mesh cs:u=0,v=1,r=0.7) {$\b{T_3}$};
  % \node (M3) at (mesh cs:u=1,v=1,r=0.35) {$\bullet$};
  \node (M4) at (mesh cs:u=1,v=0,r=0.7) {$\bullet$};

  \node (S1) at (mesh cs:u=2,v=2,r=0.7) {$\bullet$};
  \node (I2) at (mesh cs:u=1,v=2,r=0.7) {$\bullet$};
  % \node (I3) at ($(mesh cs:u=1,v=1,r=0.7)+(mesh
  % cs:u=1,v=1,r=0.35)$) {$T_3$}; 
  \node[draw,rectangle] (I4) at (mesh cs:u=2,v=1,r=0.7) {$\b{T_2}$};
  
  % \node (TP1) at (mesh cs:u=3,v=3,r=0.7) {$\bullet$};
  \node[draw,rectangle] (TS2) at (mesh cs:u=2,v=3,r=0.7) {$\b{T_4}$};
  % \node (TS3) at ($(mesh cs:u=2,v=2,r=0.7)+(mesh
  % cs:u=1,v=1,r=0.35)$) {$\bullet$}; 
  \node (TS4) at (mesh cs:u=3,v=2,r=0.7) {$\bullet$};

  % \node (S2a) at (mesh cs:u=3,v=4,r=0.7) {$\bullet$};
  % \node (S3a) at ($(mesh cs:u=3,v=3,r=0.7)+(mesh
  % cs:u=1,v=1,r=0.35)$) {$T_1$}; 
  % \node (S4a) at (mesh cs:u=4,v=3,r=0.7) {$\bullet$};

  \path[commutative diagrams/.cd, every arrow]
  (S2) edge (P1)

  (S4) edge (P1)
  (P1) edge (M2)

  (P1) edge (M4)
  (M2) edge (M1)

  (M4) edge (M1)
  (M1) edge (I2)

  (M1) edge (I4)
  (I2) edge (S1)

  (I4) edge (S1)
  (S1) edge (TS2)

  (S1) edge (TS4)

  (S2)  edge[commutative diagrams/path, dashed] (S4)
  (TS2)  edge[commutative diagrams/path, dashed] (TS4);
\end{tikzpicture}
%%% Local Variables:
%%% TeX-master: "master"
%%% End:    
  \end{center}
  Observe that $\U$ is equivalent to the cluster category of type
  $A_3$. 
  Moreover, by Thorem \ref{thm:2-cy-compatibility}(a) we have an
  isomorphism between $\End_\U(T_U)$ and $C$.
  By Theorem \ref{thm:2-cy-compatibility}(b) we have a commutative
  diagram  
  \begin{center}
    \begin{tikzcd}[column sep=huge]
      \ctilt_U\C \arrow{r}{\Hom_\C(T,-)}
      \dar[swap]{\ctred} &
      \sttilt_{\b{U}} A
      \dar{\taured} \\ 
      \ctilt \U
      \arrow[swap]{r}{\Hom_\U(T_U,-)} &
      \sttilt C
    \end{tikzcd}
  \end{center}
\end{example}

%%% Local Variables: 
%%% mode: latex
%%% TeX-master: "../master"
%%% End: 

%%%%%%%%%%%%%%%%%%%%%%%%%%%%%%%%%%%%%%%%%%%%%%%%%%%%%%%%%%%%%%%%% 

\bibliographystyle{abbrv}
\bibliography{zotero}

\begin{thebibliography}{10}

\bibitem{adachi_tau-tilting_2012}
T.~Adachi, O.~Iyama, and I.~Reiten.
\newblock $\tau$-tilting theory.
\newblock {\em {arXiv:1210.1036} (to appear in Compos. Math.)}, Oct. 2012.

\bibitem{aihara_tilting-connected_2013}
T.~Aihara.
\newblock Tilting-connected symmetric algebras.
\newblock {\em Algebr Represent Theor}, 16(3):873--894, June 2013.

\bibitem{aihara_silting_2012}
T.~Aihara and O.~Iyama.
\newblock Silting mutation in triangulated categories.
\newblock {\em J. London Math. Soc.}, 85(3):633--668, June 2012.

\bibitem{assem_elements_2006}
I.~Assem, D.~Simson, and A.~Skowro\'nski.
\newblock {\em Elements of the Representation Theory of Associative Algebras},
  volume~1 of {\em London Mathematical Society Students Texts}.
\newblock Cambridge University Press, New York City, {USA}, 2006.

\bibitem{auslander_coxeter_1979}
M.~Auslander, M.~I. Platzeck, and I.~Reiten.
\newblock Coxeter functors without diagrams.
\newblock {\em Trans. Amer. Math. Soc.}, 250:1–46, 1979.

\bibitem{auslander_applications_1991}
M.~Auslander and I.~Reiten.
\newblock Applications of contravariantly finite subcategories.
\newblock {\em Adv. Math.}, 86(1):111--152, Mar. 1991.

\bibitem{auslander_almost_1981}
M.~Auslander and S.~O. Smal\o.
\newblock Almost split sequences in subcategories.
\newblock {\em J. Algebra}, 69(2):426--454, Apr. 1981.

\bibitem{bernstein_coxeter_1973}
I.~N. Bern{\v s}te{\u\i}n, I.~M. Gel'fand, and V.~A. Ponomarev.
\newblock Coxeter functors, and {Gabriel's} theorem.
\newblock {\em Uspehi Mat. Nauk}, 28(2(170)):19--33, 1973.

\bibitem{brustle_ordered_2013}
T.~Br\"ustle and D.~Yang.
\newblock Ordered exchange graphs.
\newblock {\em {arXiv:1302.6045}}, Feb. 2013.

\bibitem{buan_tilting_2006}
A.~B. Buan, R.~Marsh, M.~Reineke, I.~Reiten, and G.~Todorov.
\newblock Tilting theory and cluster combinatorics.
\newblock {\em Adv. Math.}, 204(2):572--618, Aug. 2006.

\bibitem{buan_cluster-tilted_2007}
A.~B. Buan, R.~Marsh, and I.~Reiten.
\newblock Cluster-tilted algebras.
\newblock {\em Trans. Amer. Math. Soc.}, 359(1):323--332, 2007.

\bibitem{derksen_general_2009}
H.~Derksen and J.~Fei.
\newblock General presentations of algebras.
\newblock {\em arXiv:0911.4913}, Nov. 2009.

\bibitem{derksen_quivers_2008}
H.~Derksen, J.~Weyman, and A.~Zelevinsky.
\newblock Quivers with potentials and their representations {I: Mutations}.
\newblock {\em Selecta Math.}, 14(1):59--119, Oct. 2008.

\bibitem{geigle_perpendicular_1991}
W.~Geigle and H.~Lenzing.
\newblock Perpendicular categories with applications to representations and
  sheaves.
\newblock {\em J. Algebra}, 144(2):273--343, Dec. 1991.

\bibitem{geis_rigid_2006}
C.~Gei\ss, B.~Leclerc, and J.~Schr\"oer.
\newblock Rigid modules over preprojective algebras.
\newblock {\em Invent. Math.}, 165(3):589--632, Sept. 2006.

\bibitem{happel_partial_2005}
D.~Happel and L.~Unger.
\newblock On a partial order of tilting modules.
\newblock {\em Algebr. Represent. Theory}, 8(2):147--156, 2005.

\bibitem{happel_links_2010}
D.~Happel and L.~Unger.
\newblock Links of faithful partial tilting modules.
\newblock {\em Algebr Represent Theor}, 13(6):637--652, Dec. 2010.

\bibitem{ingalls_noncrossing_2009}
C.~Ingalls and H.~Thomas.
\newblock Noncrossing partitions and representations of quivers.
\newblock {\em Compositio Math.}, 145(06):1533--1562, 2009.

\bibitem{iyama_intermediate_2013}
O.~Iyama, P.~Jorgensen, and D.~Yang.
\newblock Intermediate co-t-structures, two-term silting objects, tau-tilting
  modules, and torsion classes.
\newblock {\em {arXiv:1311.4891}}, Nov. 2013.

\bibitem{iyama_stable_2013}
O.~Iyama and S.~Oppermann.
\newblock Stable categories of higher preprojective algebras.
\newblock {\em Advances in Mathematics}, 244:23--68, Sept. 2013.

\bibitem{iyama_silting_2013}
O.~Iyama and D.~Yang.
\newblock Silting reduction.
\newblock {\em In preparation}, 2013.

\bibitem{iyama_mutation_2008}
O.~Iyama and Y.~Yoshino.
\newblock Mutation in triangulated categories and rigid {Cohen-Macaulay}
  modules.
\newblock {\em Invent. Math.}, 172(1):117--168, Apr. 2008.

\bibitem{keller_derived_1996}
B.~Keller.
\newblock Derived categories and their uses.
\newblock In M.~Hazewinkel, editor, {\em Handbook of Algebra}, volume~1, pages
  671--701. North-Holland, 1996.

\bibitem{keller_cluster-tilted_2007}
B.~Keller and I.~Reiten.
\newblock Cluster-tilted algebras are gorenstein and stably {Calabi–Yau}.
\newblock {\em Adv. Math.}, 211(1):123--151, May 2007.

\bibitem{quillen_higher_1973}
D.~Quillen.
\newblock Higher algebraic {$K$-theory. I}.
\newblock In {\em Algebraic {$K$-theory}, I: Higher {$K$-theories} (Proc.
  Conf., Battelle Memorial Inst., Seattle, Wash., 1972)}, number 341 in Lecture
  Notes in Math., pages 85--147. Springer, Berlin, 1973.

\bibitem{riedtmann_simplicial_1991}
C.~Riedtmann and A.~Schofield.
\newblock On a simplicial complex associated with tilting modules.
\newblock {\em Comment. Math. Helv.}, 66:70--78, 1991.

\bibitem{ringel_self-injective_2008}
C.~M. Ringel.
\newblock The self-injective cluster-tilted algebras.
\newblock {\em Arch. Math.}, 91(3):218--225, Sept. 2008.

\bibitem{schiffler_geometric_2008}
R.~Schiffler.
\newblock A geometric model for cluster categories of type {$D_n$}.
\newblock {\em J. Algebraic Combin.}, 27(1):1--21, Feb. 2008.

\bibitem{wei_semi-tilting_2012}
J.~Wei.
\newblock Semi-tilting complexes.
\newblock {\em Isr. J. Math.}, pages 1--23, 2012.

\end{thebibliography}

\end{document}